\documentclass[11 pt]{amsart}
\usepackage{tikz}
\usepackage{amscd,amsfonts,amssymb,amsmath}
\usepackage{hyperref}
\usepackage{epsfig}
\newtheorem{theorem}{Theorem}[section]
\newtheorem{corollary}[theorem]{Corollary}
\newtheorem{lemma}[theorem]{Lemma}
\newtheorem{proposition}[theorem]{Proposition}
\theoremstyle{definition}

\newtheorem{problem}[theorem]{Problem}

\newtheorem{example}[theorem]{Example}
\newtheorem{remark}[theorem]{Remark}
\numberwithin{equation}{subsection}

\newcommand{\Core}{\operatorname{Core}}

\newcommand{\Aut}{\operatorname{Aut}}

\newcommand{\Conj}{\operatorname{Conj}}

\newcommand{\Inn}{\operatorname{Inn}}

\newcommand{\R}{\operatorname{R}}

\begin{document}

\title{General constructions of biquandles and their symmetries}

\author{Valeriy Bardakov\and Timur Nasybullov\and Mahender Singh}

\address{Tomsk State University, pr. Lenina 36, 634050 Tomsk, Russia, Sobolev Institute of Mathematics, Acad. Koptyug avenue 4, 630090 Novosibirsk, Russia, Novosibirsk State University, Pirogova~1, 630090 Novosibirsk, Russia, Novosibirsk State Agricultural University, Dobrolyubova 160, 630039 Novosibirsk, Russia.}
\email{bardakov@math.nsc.ru}

\address{Katholieke Universiteit Leuven KULAK, 53  E.~Sabbelaan, 8500, Kortrijk, Belgium.}
\email{timur.nasybullov@mail.ru}

\address{Department of mathematical sciences, Indian Institute of Science Education and Research (IISER) Mohali, Sector 81,  S. A. S. Nagar, P. O. Manauli, 140306, Punjab, India.}
\email{mahender@iisermohali.ac.in}

\keywords{Quandle, biquandle, quandle covering, automorphism, virtual knot, knot invariant, Yang-Baxter equation}
\subjclass[2010]{57M25, 57M27, 20N05, 	16T25}

\begin{abstract}
Biquandles are algebraic objects with two binary operations whose axioms encode the generalized Reidemeister moves for virtual knots and links. These objects also provide set theoretic solutions of the well-known Yang-Baxter equation. The first half of this paper proposes some natural constructions of biquandles from groups and from their simpler counterparts, namely, quandles. We completely determine all words in the free group on two generators that give rise to (bi)quandle structures on all groups. We give some novel constructions of biquandles on unions and products of quandles, including what we refer as the holomorph biquandle of a quandle. These constructions give a wealth of solutions of the Yang-Baxter equation. We also  show that for nice quandle coverings a biquandle structure on the base can be lifted to a biquandle structure on the covering. In the second half of the paper, we determine automorphism groups of these biquandles in terms of associated quandles showing elegant relationships between the symmetries of the underlying structures.
\end{abstract}

\maketitle

\section{Introduction}\label{sec-intro}
 A quandle is an algebraic system with a single binary operation satisfying three axioms that are algebraic analogues of the three Reidemeister moves of diagrams of knots in the $3$-sphere. Such algebraic systems were introduced independently by Joyce \cite{Joy} and Matveev \cite{Mat} as invariants for knots. More precisely,  to each oriented diagram $D_K$ of an oriented knot $K$ in the $3$-sphere one can associate a quandle $Q(K)$ which does not change on applying Reidemeister moves to the diagram $D_K$.  Joyce and Matveev proved that two knot quandles $Q(K_1)$ and $Q(K_2)$ are isomorphic if and only if $K_1$ and $K_2$ are weakly equivalent, i.~e. there exists a homeomorphism of the $3$-sphere which maps $K_1$ onto $K_2$.  Over the years, quandles have been investigated by various authors for constructing newer invariants for knots and links (see, for example, \cite{Carter, CatNas, Fenn-Rourke, Kamada20122, NanSinSin, Nelson}). In order to obtain reasonably strong knot invariants from quandles, it is necessary to understand them from algebraic point of view. Algebraic properties of quandles including their automorphisms and residual properties have been investigated, for example, in \cite{BDS, BNS, BSS, BiaBon, Clark, HosSha, JPSZ, NelWon, Nosaka}. A (co)homology theory for quandles and racks has been developed in \cite{Carter2, Fenn1, Fenn2, Nosaka2013}, which, as applications has led to stronger invariants for knots and links.  Many new constructions of quandles have been introduced, for example, in \cite{BarNas, BonCraWhi, Crans-Nelson}.

Virtual knot theory was introduced by Kauffman in \cite{Kauff}. In the same paper Kauffman extended several known knot invariants from classical knot theory to virtual knot theory, in particular, he defined the quandle $Q(K)$ of a virtual knot $K$. Despite the fact that the fundamental quandle is an almost complete invariant for classical links, its extension to virtual links is a relatively weak invariant. This led Fenn, Jordan-Santana and Kauffman \cite{Fenn} to introduce the notion of a biquandle which generalizes the notion of a quandle, and gives a powerful invariant for virtual links. We note that the idea of a birack first appeared in \cite{Fenn-Rourke-Sanderson}. A still unresolved conjecture from \cite{Fenn} states that the knot biquandle is in some sense a complete invariant for virtual knots just as the knot quandle is in a sense a complete invariant for classical knots. It is unknown whether there exist classical or surface knots which have isomorphic fundamental quandles but distinct fundamental biquandles. Ashihara \cite{Ashihara} proved that two ribbon $2$-knots or ribbon torus-knots with isomorphic fundamental quandles have isomorphic fundamental biquandles.

Besides their importance in constructing invariants for virtual knots and links, biquandles can be used for constructing representations of (virtual) braid groups \cite{BarNas2, Fenn} and solutions of the well-known Yang-Baxter equation \cite{CarElh} which has wide applications  beyond mathematics (for other algebraic systems which give solutions of the Yang-Baxter equation, see, for example, \cite{GuaVen, Nas, Rump}). 

All applications of biquandles mentioned above require an extensive set of examples of these structures, and  hence it is desirable to have some canonical constructions of biquandles. In a recent preprint \cite[Theorem 3.2]{Horvat}, it is shown that every biquandle $B$ can be constructed from some quandle $\mathcal{Q}(B)$ (called the associated quandle of $B$) using a specific family of automorphisms of $\mathcal{Q}(B)$ (called a biquandle structure). The idea is used further in \cite{Horvat-Crans} to construct biquandles on the set of homomorphisms between two given biquandles. Note that it is quite difficult to construct a non-trivial biquandle using  \cite[Theorem 3.2]{Horvat} since there seems no natural way to find a non-trivial biquandle structure on a given quandle.
 
The purpose of this paper is to give natural constructions of biquandles from groups and from their simpler counterparts, namely, quandles. We give some novel constructions of biquandles on unions and products of quandles, including what we call the holomorph biquandle of a given quandle. These constructions give a wealth of solutions of the Yang-Baxter equation. Analogous to the covering space theory for topological spaces, Eisermann \cite{Eisermann} developed a theory of quandle coverings which he used to study knot invariants. We also  show that for some good quandle coverings a biquandle structure on the base can be lifted to a biquandle structure on the covering. Since the fundamental biquandles in some sense encode symmetries of virtual knots, its worthwhile to explore symmetries of other biquandles. We determine automorphism groups of these biquandles in terms of associated quandles, and our results exhibit elegant relationships between the symmetries of the underlying structures. 

The paper is organised as follows. In Section \ref{sec-prelim}, we recall necessary preliminaries on quandles and biquandles which also set notations and conventions for the rest of the paper. We give some new  generalizations of Alexander and dihedral biquandles  (Proposition \ref{gen-dihedral-biquandle}). In Section \ref{sec-verbal}, we determine all words in the free group on two generators that give rise to quandle and biquandle structures on all groups. A complete classification of words yielding quandle structures is given in Proposition~\ref{p7.2}, and that of pair of words yielding biquandle structures is given in Theorem~\ref{ver-biquandle-thm}. In Section \ref{newconstructionsnew}, we give new constructions of biquandles from quandles. We construct biquandle structures on unions of two quandles that act on each other via quandle automorphisms (Theorem~\ref{gen-biquandle-union}). As an application of this construction to virtual knot theory, we give an example of a biquandle structure on a trivial quandle whose associated biquandle can distinguish different virtual links with the same number of components, whereas, the trivial quandle obviously cannot distinguish any two links with the same number of components (Example \ref{trivial-biquandle-distinguish}). We construct biquandle structures on product of two quandles (Theorem \ref{action-product-biquandle}), which, in particular, yields a biquandle what we refer as the holomorph biquandle (Remark \ref{rem-holomorph}). We also prove that for nice quandle coverings a biquandle structure on the base can be lifted to a biquandle structure on the covering (Theorem \ref{lifting theorem}). It has been an open question to  give an explicit model for a free biquandle \cite{Fenn}. We believe that a free biquandle should be obtainable using a biquandle structure on a free quandle, and prove that every biquandle structure on the trivial quandle $T_n$ can be lifted to a biquandle structure on the free quandle $FQ_n$ (Proposition~\ref{lifting}). In Section \ref{sec-automorphisms}, we investigate structure of automorphism groups of quandles constructed in the preceding sections. We begin with results on automorphisms of  generalised dihedral and Alexander biquandles  (Proposition \ref{auto-gen-dihedral-biquandle} and Proposition \ref{auto-gen-alexander}). For constant actions case, we completely determine the automorphism group of the union biquandle in terms of automorphism groups of underlying quandles (Theorem~\ref{auto-union-biquandle}). Subsection \ref{sec-product-biquandle} gives a detailed analysis of automorphism groups of biquandles on products of quandles (Proposition~\ref{auto-product-subgroup}, Proposition~\ref{shsequence} and Theorem~\ref{automconprod}). As a consequence, it is shown that the automorphism group of the holomorph biquandle coincides with the automorphism group of the underlying quandle (Corollary \ref{authol}), which, in turn shows that there exists a sequence of finite biquandles $B_1,B_2,\dots$ such that $|B_k|\to\infty$ but $|{\rm Aut}(B_k)|/|B_k|\to 0$ (Corollary \ref{seq-biquandles}). The paper concludes by relating the automorphism groups of the associated biquandle of a quandle and that of its covering quandle (Proposition \ref{biquandle-auto-in-lifting}). Several open problems are formulated throughout the paper.

\section{Quandles and biquandles}\label{sec-prelim}
In this section we give necessary preliminaries on quandles and biquandles.

\subsection{Racks and quandles}\label{quandlesprem}
A \textit{rack} $R$ is an algebraic system with one binary algebraic operation $(x,y)\mapsto x*y$ which satisfies the following two axioms:
\begin{itemize}
\item[(r1)] the map $S_x:y\mapsto y*x$ is a bijection of $R$ for all $x\in R$,
\item[(r2)] $(x*{y})*z=(x*z)*({y*z})$ for all $x,y,z\in R$.
\end{itemize}
Axioms (r1) and (r2) imply that the map $S_x$ is an automorphism of $R$ for all $x\in R$ . A rack $R$ is said to be \textit{involutory} if $S_x^2=id$ for all $x\in R$. A rack $R$ is called \textit{faithful} if the map $x\mapsto S_x$ is injective. The group ${\rm Inn}(R)=\langle S_x~|~x\in R\rangle$  generated by all $S_x$ for $x\in R$ is called  the \textit{group of inner automorphisms} of $R$. From axioms (r2) follows that the equality
\begin{equation}\label{innerrules}
S_{x*y}=S_yS_xS_y^{-1}
\end{equation}
holds for all $x,y\in Q$. The group ${\rm Inn}(R)$ acts on $R$ in the natural way. The orbit of an element $x\in R$ under this action is denoted by ${\rm Orb}(x)$ and is called the \textit{orbit} of $x$. If $R={\rm Orb}(x)$ for some $x\in R$, then $R$ is said to be \textit{connected}. For the sake of simplicity, for elements $x,y\in R$, we denote by $y*^{-1}x$ the element $S_x^{-1}(y)$, and sometimes we denote by $y*^1x=y*x$. Generally, the operation $*$ is not associative. For elements
$x_1,\dots,x_n\in R$ and integers $\varepsilon_2,\dots,\varepsilon_n\in\{\pm1\}$ we denote by
$$x_1*^{\varepsilon_2}x_2*^{\varepsilon_3}x_3*^{\varepsilon_4}\dots *^{\varepsilon_n}x_n=((\cdots((x_1*^{\varepsilon_2}x_2)*^{\varepsilon_3}x_3)*^{\varepsilon_4}\cdots)*^{\varepsilon_n}x_n).$$
	A rack $R$ which satisfies the additional axiom
\begin{itemize}
\item[(q1)] $x*x=x$ for all $x\in R$
\end{itemize}
is called a \textit{quandle}. The simplest example of a quandle is the trivial quandle on a set $X$, that is the quandle $Q=(X,*)$, where $x*y=x$ for all $x,y\in X$. If $|X|=n$, then the trivial quandle on $X$ is denoted by $T_n$. A lot of examples of quandles come from groups. 
\begin{example}\label{conjugationex}Let $G$ be a group. For elements $x,y\in G$, denote by $x^y=y^{-1}xy$ the conjugate of $x$ by $y$. For an arbitrary integer $n$ the set $G$ with the operation $x*y=x^{y^n}=y^{-n}xy^n$ forms a quandle. This quandle is called the \textit{$n$-th conjugation quandle of the group $G$} and is denoted by ${\rm Conj}_n(G)$. For the sake of simplicity we denote the first conjugation quandle ${\rm Conj}_1(G)$ by ${\rm Conj}(G)$. 
\end{example}
It can be seen that ${\rm Conj}:\textbf{ \text{Grp}}\to \textbf{\text{Qnd}}$ given by $G\mapsto{\rm Conj}(G)$ is a functor from the category \textbf{Grp} of groups to the category \textbf{Qnd} of quandles. 
\begin{example}\label{coreex}
Let $G$ be a group, and $x*y=yx^{-1}y$ for all $x,y\in G$. Then $G$ with  the operation $*$ is a quandle, which is called the \textit{core quandle of a group $G$} and is denoted by $\Core(G)$. In particular, if $G$ is an abelian group, then the quandle $\Core(G)$ is called the \textit{Takasaki quandle of the abelian group $G$} and is denoted by ${\rm T}(G)$. Such quandles were studied by Takasaki in \cite{Takasaki}. If $G=\mathbb{Z}_n$ is the cyclic group of order $n$, then the Takasaki quandle ${\rm T}(\mathbb{Z}_n)$ is called the \textit{dihedral quandle on $n$ elements} and is denoted by ${\rm R}_n$. If $n$ is odd, then ${\rm R}_n$ is a faithful connected quandle.
\end{example}
\begin{example}
If $\varphi$ is an automorphism of a group $G$, then $G$ with the operation $x*y=\varphi(xy^{-1})y$ forms a quandle which is denoted by ${\rm Alex}(G,\varphi)$. Such quandles are referred in the literature as {\it generalized Alexander quandles}.  Alexander quandles were studied, for example, in \cite{BDS, Clark1, Clark}. 
\end{example}

For more examples of quandles arising from groups see, for example, \cite{BarNas}.

For a quandle $Q$ denote by ${\rm Adj}(Q)$ the group with the set of generators $Q$ and the set of relations 
$$x*y=y x y^{-1}$$ 
for all $x,y\in Q$. The group ${\rm Adj}(Q)$ is called the \textit{adjoint group} or the \textit{enveloping group} of the quandle $Q$. From equality (\ref{innerrules}) follows that the map $\xi$ which maps the generator $x\in Q$ of ${\rm Adj}(Q)$ to the generator $S_x$ of ${\rm Inn}(Q)$ is a group homomorphism 
$$\xi:{\rm Adj}(Q)\to {\rm Inn}(Q).$$ 
Thus, the group ${\rm Adj}(Q)$ acts on $Q$ by the rule: 
$$g\cdot x=\xi(g)(x)$$ 
for $g\in {\rm Adj}(Q)$, $x\in Q$. We can see that ${\rm Adj}: \textbf{Qnd} \to  \textbf{Grp}$ given by $Q\mapsto {\rm Adj}(Q)$ is a functor from the category \textbf{Qnd} of quandles to the category \textbf{Grp} of groups. Moreover, this functor  is the left adjoint of the functor ${\rm Conj}: \textbf{Grp} \to \textbf{Qnd}$.

\textit{The free quandle} on a set $X\neq\varnothing$ is a quandle  $FQ(X)$ together with a map $\varphi_Q : X \to FQ(X)$  such that for every  map $\rho : X \to Q$, where $Q$ is a quandle, there exists a
unique quandle homomorphism $\overline{\rho} : FQ(X) \to Q$ such that $\overline{\rho}\varphi_Q = \rho$. The free quandle is unique up to isomorphism. The following construction of the free quandle $FQ(X)$ on the set $X$ of generators is introduced in \cite{Fenn-Rourke, Kamada2017} (see also~\cite{BarNas}). Let $F(X)$ be the free group with the free generators $X$. On the set $X\times F(X)$ denote by $\sim$ the equivalence relation $(a,w)\sim(a,aw)$ for $a\in X$, $w\in F(X)$. On the set of equivalence classes $X\times F(X)/_{\sim}$ define the operation
$$ [(a, u)] * [(b, v)] = [(a, u v^{-1} b v)]$$
for $a, b \in X$, $u, v \in F(X)$. Here $[(a, u)]$ denotes the equivalence class of $(a,u)$. The set $X\times F(X)/_{\sim}$ with the operation $*$ is the free quandle on the set $X$.  The free quandle $FQ(X)$ is a subquandle of ${\rm Conj}(F(X))$ consisting of conjugacy classes of the free generators $X$ of $F(X)$  \cite[Theorem 4.1]{Joy}. If $X$ has $n$ elements, then the free quandle $FQ(X)$ is denoted by $FQ_n$.

Quandles were introduced in \cite{Joy, Mat} as an invariant for classical links. Kauffman \cite{Kauff} extended this invariant  to virtual knots and links. Let $K$ be a virtual link, and $D_K$ be a diagram of $K$. A strand of $D_K$ going from one crossing (classical or virtual) to another crossing (classical or virtual) is called an \textit{arc} of $D_K$. The \textit{fundamental quandle} $Q(K)$ of a virtual link $K$  can be found from the knot diagram $D_K$ of $K$ in the following way: the set of generators of $Q(K)$ is the set of arcs of $D_K$; the set of relations of $Q(K)$ is the set of equalities which can be written from the crossings of $D_K$ in the way depicted on Figure~\ref{quandlerelations}.
\begin{figure}[hbt!]
\noindent\centering{
\includegraphics[height=25mm]{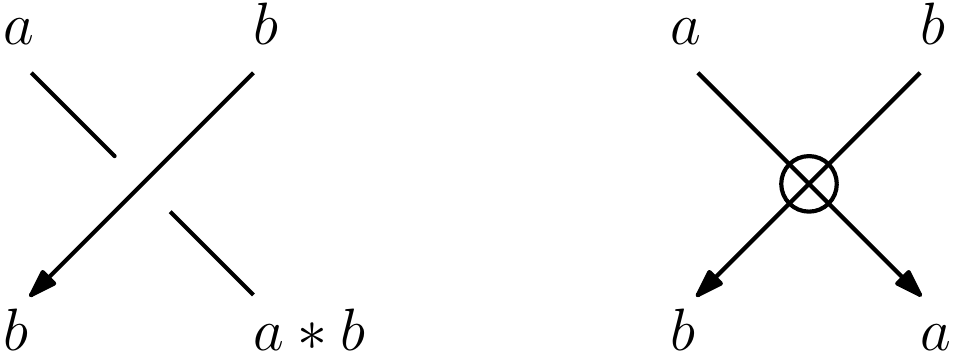}}
\caption{Labels of arcs in $D_K$.}
\label{quandlerelations}
\end{figure}

Let $Q$ be a finite quandle. A \textit{coloring of a diagram} $D_K$ of a link $K$ by elements of $Q$ is a labeling of arcs of $D_K$ by elements of $Q$. A coloring is said to be \textit{proper} if in the neighborhood of all crossing the labels of arcs are as on Figure~\ref{quandlerelations}. The number of proper colorings of $D_K$ by elements of $Q$ is an invariant of $K$ which is called the \textit{quandle-coloring invariant defined by $Q$} and is denoted by $C_Q(K)$. Note that the quandle coloring of $D_K$ by only one element of $Q$ is always proper, therefore $C_Q(K)\geq |Q|$. The number $C_Q(K)$ is equal to the number of homomorphisms from the fundamental quandle $Q(K)$ to $Q$. If $Q=T_n$ is a trivial quandle with $n$ elements, then from Figure~\ref{quandlerelations} it is clear that $C_Q(K)=n^k$, where $k$ is the number of components of $K$. Thus, the quandle-coloring invariant defined by the trivial quandle cannot distinguish any two links with the same number of components.

\subsection{Biracks and biquandles}\label{biquandlesprem}
A \textit{biquandle} $B$ is an algebraic system with two binary algebraic operations $(x,y)\mapsto x\underline{*}y$, $(x,y)\mapsto x\overline{*}y$ which satisfy the following axioms:
\begin{enumerate}
\item $x\underline{*}x=x\overline{*}x$ for all $x\in B$,
\item the maps $\alpha_y,\beta_y:B\to B$ and $S:B\times B\to B\times B$ given by $\alpha_y(x)=x\underline{*}y$, $\beta_y(x)=x\overline{*}y$, $S(x,y)=(y\overline{*}x,x\underline{*}y)$ are bijections for all $y\in B$,
\item the equalities
\begin{enumerate}
\item $(x\underline{*}y)\underline{*}(z\underline{*}y)=(x\underline{*}z)\underline{*}(y\overline{*}z)$,
\item $(x\underline{*}y)\overline{*}(z\underline{*}y)=(x\overline{*}z)\underline{*}(y\overline{*}z)$,
\item $(x\overline{*}y)\overline{*}(z\overline{*}y)=(x\overline{*}z)\overline{*}(y\underline{*}z)$
\end{enumerate}
hold for all $x,y,z\in B$.
\end{enumerate}
The algebraic system $(X,\underline{*},\overline{*})$ which satisfies only the second and the third axioms of biquandles is called a \textit{birack}. For the sake of simplicity we denote by $x\underline{*}^{-1}y=\alpha_y^{-1}(x)$, $x\overline{*}^{-1}y=\beta_y^{-1}(x)$. The operations $\underline{*}$, $\overline{*}$ are not associative in general. In order to avoid large number of parentheses, for elements $x_1,x_2,\dots,x_n\in B$ and operations $\circ_2,\circ_3,\dots,\circ_n\in\{\underline{*},\underline{*}^{-1},\overline{*},\overline{*}^{-1}\}$, we denote by
$$x_1\circ_2x_2\circ_3x_3\circ_4\dots \circ_nx_n=((\cdots((x_1\circ_2x_2)\circ_3x_3)\circ_4\cdots)\circ_nx_n).$$
A biquandle $B$ is said to be \textit{involutory} if the equalities
\begin{align}\label{inv-biquandle-condition}
x \underline{*} (y \overline{*} x)= x \underline{*} y,&& x \overline{*} (y \underline{*} x)= x \overline{*} y,&& (x \underline{*} y) \underline{*} y=x,&& (x \overline{*} y) \overline{*} y=x
\end{align}
hold for all $x, y \in B$.

If $Q=(X,*)$ is a quandle, then the algebraic system $B=(X,\underline{*},\overline{*})$ with $x\underline{*}y=x*y$, $x\overline{*}y=x$ for all $x,y\in X$ is a biquandle which is denoted by $\mathcal{B}(Q)$. Thus, the notion of a biquandle generalizes the notion of a quandle. Moreover, $\mathcal{B}:\textbf{\text{Qnd}}\to\textbf{\text{Bqnd}}$ given by $Q\mapsto \mathcal{B}(Q)$ is a functor from the category \textbf{Qnd} of quandles to the category \textbf{Bqnd} of biquandles. 
\begin{problem} What is the adjoint to the functor $\mathcal{B}$?
\end{problem}
If $B=(X,\underline{*},\overline{*})$ is a biquandle, then the map $r:B\times B \to B\times B $ given by $r(a,b\overline{*}a)=(b,a\underline{*}b)$
for $a,b\in B$ satisfies the following operator equation on $B\times B \times B$
$$(r\times id)(id \times r)(r\times id)=(id\times r)(r \times id)(id\times r)$$
which is known as the Yang-Baxter equation \cite{EtiSchSol} and has applications well beyond mathematics. The pair $(B,r)$ in this situation is called the set theoretical solution of the Yang-Baxter equation. In order to construct new set-theoretical solutions of the Yang-Baxter equation one should construct new examples of biquandles. Let us give some examples of biquandles with both operations $\underline{*},\overline{*}$ non-trivial.

\begin{example}\label{wadaex}
Let $G$ be a group, and $x\underline{*}y=y^{-1}x^{-1}y$, $x\overline{*}y=y^{-2}x$ for $x,y\in G$. Then $(G,\underline{*},\overline{*})$ is a biquandle which is called the \textit{Wada biquandle} (see \cite{FenBar,Wada}). 
\end{example}
\begin{example} If $K$ is an integral domain, $t,s$ are invertible elements of $K$, and $M$ is a left $K$-module, then the set $M$ with the operations $x\underline{*}y=tx+(s-t)y$, $x\overline{*}y=sx$ for $x,y\in M$ is a biquandle. If $K=\mathbb{Z}[t^{\pm1},s^{\pm1}]$, then  this biquandle is called the \textit{Alexander biquandle} and is denoted by ${\rm A}_{s,t}(M)$. If $K=\mathbb{Z}_n$, $M=K$ is the additive group of $K$, and $t=-1$, then this biquandle is called the \textit{dihedral biquandle} of order $n$. Alexander biquandles and dihedral biquandles were studied, for example, in \cite{CraHenNel, Horvat, LamNel, Mur}. 
\end{example}

Recall that an automorphism $\phi$ of a group $G$ is said to be \textit{central} if $x^{-1}\phi(x)$ belongs to the center of $G$ for all $x\in G$. The following proposition gives generalizations of Alexander and dihedral biquandles which we did not see anywhere in literature.

\begin{proposition}\label{gen-dihedral-biquandle}
Let $G$ be a group. 
\begin{enumerate}
\item If $\phi$ is a central automorphism of $G$, then $G$ with operations $x \underline{*} y= \phi(y) x^{-1} y$ and $x \overline{*} y= \phi(x)$ for $x, y \in G$ is a biquandle. This biquandle is denoted by $B(G,\phi)$.
\item If $\phi, \psi$ is a pair of commuting automorphisms of $G$, then $G$ with operations $x \underline{*}y =\phi(xy^{-1})\psi(y)$, $x \overline{*} y= \psi(x)$ for $x, y \in G$ is a biquandle. This biquandle is denoted by ${\rm A}_{\psi,\phi}(G)$.
\end{enumerate}
\end{proposition}
\begin{proof} A direct check.
\end{proof}
Biquandles were introduced in \cite{Fenn} as tools for constructing  invariants for virtual links. Let $B$ be a finite biquandle. A \textit{coloring of a diagram} $D_K$ of a link $K$ by elements of $B$ is a labeling of arcs of $D_K$ by elements of $B$. A coloring is said to be \textit{proper} if in the neighborhood of all crossing the labels of arcs are as on Figure~\ref{biquandlerelations2}. 
\begin{figure}[hbt!]
\noindent\centering{
\includegraphics[height=25mm]{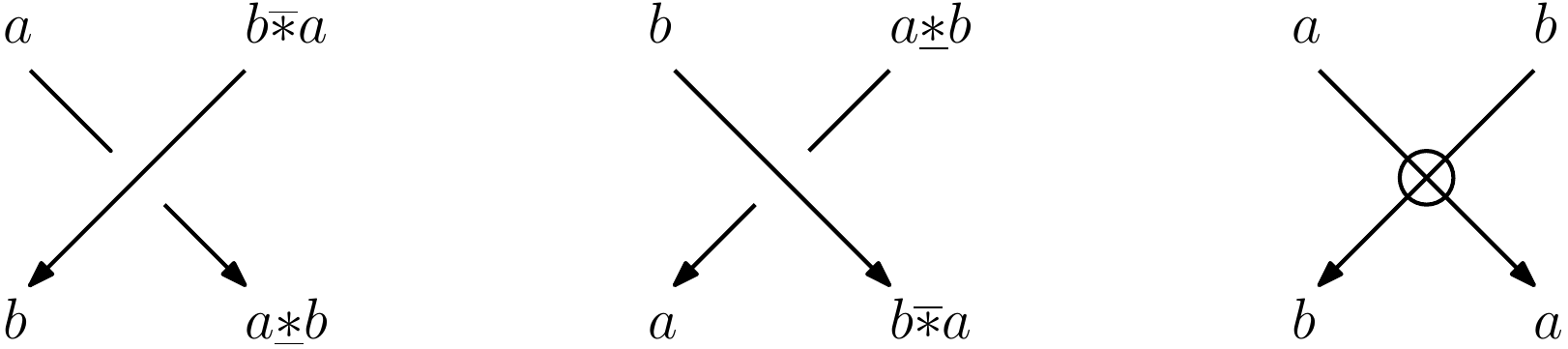}}
\caption{Labels of arcs in $D_K$.}
\label{biquandlerelations2}
\end{figure}

The number of proper colorings of $D_K$ by elements of $B$ is an invariant of $K$ which is called the \textit{biquandle-coloring invariant defined by $B$} and is denoted by $C_B(K)$. If $B=\mathcal{B}(Q)$, where $Q$ is a finite quandle, then $C_B(K)=C_Q(K)$. Note that for biquandle colorings the coloring by only one element of $B$ is not necessarily proper, so, there are no connections between $C_B(K)$ and $|B|$. 

\begin{remark}Note that in paper \cite{Fenn} where biquandles were originally defined there is another definition of biquandles. According to \cite{Fenn}, a biquandle is again an algebraic system with two binary algebraic operations $(x,y)\mapsto x^y$, $(x,y)\mapsto x_y$, but the axioms for these operations differ from the axioms for operations $\underline{*}, \overline{*}$ introduced at the beginning of this section.  It happened since axioms for operations $\underline{*}, \overline{*}$ come from Figure~\ref{biquandlerelations2}, while the axioms for operations $(x,y)\mapsto x^y$, $(x,y)\mapsto x_y$ in \cite{Fenn} come from Figure~\ref{oldbiquandlerelations}.
\begin{figure}[hbt!]
\noindent\centering{
\includegraphics[height=27mm]{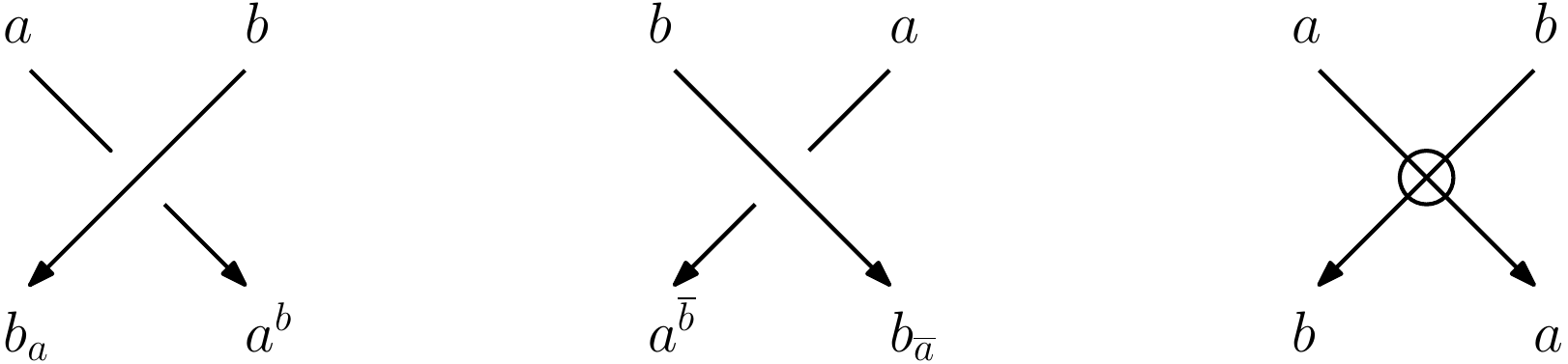}}
\caption{Labels of arcs due to \cite{Fenn}.}
\label{oldbiquandlerelations}
\end{figure}

Comparing Figure~\ref{biquandlerelations2} and Figure~\ref{oldbiquandlerelations} we see that operations $(x,y)\mapsto x^y$, $(x,y)\mapsto x_y$ can be written in terms of operations $\underline{*}, \overline{*}$ in the following way 
\begin{align}
\notag x^y=x\underline{*}(y\overline{*}^{-1}x),&& y_x=y\overline{*}^{-1}x.
\end{align}
Thus, biquandles defined in \cite{Fenn} and biquandles defined at the beginning of this section are in one-to-one correspondance. In this paper we use the definition of a biquandle formulated at the beginning of this section since we use some ideas from the paper \cite{Horvat} which is formulated in these terms.
\end{remark}

\section{Verbal quandles and biquandles}\label{sec-verbal}
Let $w = w(x,y)$ be an element of the free group $F(x,y)$ with two generators. If $G$ is a group, then $w$ defines a map $w:G\times G\to G$. For $g,h\in G$ denote by $g *_w h = w(g, h)$. We can think about~$*_w$ as about new operation on the group $G$, so, $(G,*_w)$ is an algebraic system.  
The case when~$*_w$ defines a group is considered, for example, in \cite{BarSim, Coo, Sol} (see also \cite[Problem~6.47]{Kourovka}). If the algebraic system $(G,*_w)$ is a rack (respectively, quandle), then we call this rack (respectively, quandle) a {\it verbal rack} (respectively, \textit{verbal quandle}) defined by the word $w$. Similarly, if we have two elements $u=u(x,y)$, $v=v(x,y)$ from the free group $F(x,y)$, then for an arbitrary group $G$ these words define two operations $g\overline{*}h=u(g,h)$, $g\underline{*} h= v(g,h)$ for $g,h\in G$. If the resulting algebraic system is a birack (respectively, biquandle), then we call this birack (respectively, biquandle) a \textit{verbal birack} (respectively, \textit{verbal biquandle}) defined by the words $u,v$.

In this section we investigate verbal quandles and biquandles. 

\subsection{Racks and quandles}
As we noticed in Section~\ref{quandlesprem}, if $w(x,y) = y^{-n} x y^n$ for an integer $n$, then $(G, *_w)$ is the $n$-th conjugation quandle ${\rm Conj}_n(G)$ (Example~\ref{conjugationex}), and if $w(x,y) = y x^{-1} y$, then $(G, *_w)$ is the core quandle ${\rm Core}(G)$ (Example~\ref{coreex}). In this section we find all possible words $w\in F(x,y)$ such that the algebraic system $(G,*_w)$ is a rack (respectively, quandle) for all groups~$G$. 

\begin{proposition}\label{p7.2}
Let $w=w(x,y)\in F(x,y)$ be such that $Q = (G, *_{w})$ is a rack for every group $G$. Then $Q$ is a quandle, and $w(x,y) = y x^{-1} y$ or $w(x,y) = y^{-n} x y^n$ for some $n \in \mathbb{Z}$.
\end{proposition}
\begin{proof}Let $w = x^{\alpha_1} y^{\beta_1} \ldots x^{\alpha_k} y^{\beta_k}$ for $\alpha_i, \beta_i \in \mathbb{Z}$ be a reduced word, where all $\alpha_i, \beta_i$ are non-zero with possible exceptions for $\alpha_1$ and $\beta_k$. Since $(G,*_w)$ is a rack, from the axiom (r1) of a rack follows that for every $a, b \in G$ there exists an element $c \in G$ such that 
$$c^{\alpha_1} a^{\beta_1} \ldots c^{\alpha_k} a^{\beta_k}=c *_w a = b,$$ 
but it is possible if and only if $w =  y^{\alpha}  x^{\varepsilon} y^{\beta}$ for $\alpha, \beta \in \mathbb{Z}$, $\varepsilon \in \{\pm 1 \}$. From the axiom (r2) of a rack for all $x,y,z\in G$, we obtain $
(x *_w y) *_w z = (x *_w z) *_w(y *_w z)$, which can be rewritten as
\begin{equation}\label{a3}
z^{\alpha} (y^{\alpha} x^{\varepsilon} y^{\beta})^{\varepsilon} z^{\beta} = (z^{\alpha} y^{\varepsilon} z^{\beta})^{\alpha} (z^{\alpha} x^{\varepsilon} z^{\beta})^{\varepsilon} (z^{\alpha} y^{\varepsilon} z^{\beta})^{\beta}.
\end{equation}
If in equality (\ref{a3}) we put $x=y=1$, then we obtain
$$
z^{(\alpha + \beta)} = z^{(\alpha + \beta) (\alpha + \beta + \varepsilon)},
$$
and hence $\alpha + \beta = 0$ or $\alpha + \beta \not = 0$ and $1 = \alpha + \beta + \varepsilon$, i.~e. $\varepsilon = -1$.

\textit{Case 1}: $\alpha + \beta = 0$. In this case equality (\ref{a3}) can be rewritten in the form
$$
z^{\alpha} y^{\alpha} x y^{-\alpha} z^{-\alpha} = z^{\alpha} y^{\alpha \varepsilon} x y^{-\alpha \varepsilon} z^{-\alpha}.
$$
Since this equality holds, in particular, in the free group $F(x, y, z)$ on $3$ generators, we have $\varepsilon = 1$ and $w = y^{\alpha} x y^{-\alpha}$.

{\it Case} 2: $\alpha + \beta = 2$. In this case $\varepsilon = -1$, $\beta = 2 - \alpha$, and equality (\ref{a3}) can be rewritten in the form
\begin{equation}\label{case2verbal}
z^{\alpha} y^{\alpha-2} x y^{-\alpha} z^{2-\alpha} = (z^{\alpha} y^{-1} z^{2-\alpha})^{\alpha}
z^{\alpha-2} x z^{-\alpha} (z^{\alpha} y^{-1} z^{2-\alpha})^{2-\alpha}.
\end{equation}
Since this equality must hold in every group for all $x,y,z$, in particular, it holds in the free group $F(x, y, z)$ on $3$ generators. Depending on $\alpha$, we have the following subcases.

{\it Case} 2.1: $\alpha<0$. Equality (\ref{case2verbal}) in this case implies that
$$
z^{\alpha} y^{\alpha-2} x y^{-\alpha} z^{2-\alpha} = \underbrace{(z^{\alpha-2} y z^{-\alpha}) \ldots (z^{\alpha-2} y z^{-\alpha})}_{-\alpha~ \mbox{\tiny times}} z^{\alpha-2}  x z^{-\alpha}
\underbrace{(z^{\alpha} y^{-1} z^{2-\alpha}) \ldots (z^{\alpha} y^{-1} z^{2-\alpha})}_{2-\alpha~ \mbox{\tiny times}}.  
$$
The left side of this equality is reduced, while the word on the right side of this equality after cancellations starts with $z^{\alpha-2} y z^{-2}$. This equality cannot hold in $F(x,y,z)$, so, this case is impossible.

{\it Case} 2.2: $\alpha = 0$. Equality (\ref{case2verbal}) in this case implies that $y^{-2} x = z^{-2} x y^{-1} z^{2} y^{-1}$, but it is not true in $F(x,y,z)$. So, this case is impossible.

{\it Case} 2.3: $\alpha = 1$. In this case $w = y x^{-1} y$ and equality (\ref{case2verbal}) obviously holds.

{\it Case} 2.4: $\alpha = 2$. Equality (\ref{case2verbal}) in this case implies that $x y^{-2} = y^{-1} z^{2} y^{-1} x z^{-2}$, but it is not true in $F(x,y,z)$, so, this case is impossible.

{\it Case} 2.5: $\alpha >2$. In this case $2-\alpha < 0$, and equality (\ref{case2verbal}) implies that
$$
z^{\alpha} y^{\alpha-2} x y^{-\alpha} z^{2-\alpha} = \underbrace{(z^{\alpha} y^{-1} z^{2-\alpha}) \ldots (z^{\alpha} y^{-1} z^{2-\alpha})}_{\alpha~ \mbox{\tiny times}}
z^{\alpha-2}  x z^{-\alpha}
\underbrace{(z^{\alpha-2} y z^{-\alpha}) \ldots (z^{\alpha-2} y z^{-\alpha})}_{\alpha-2~ \mbox{\tiny times}}.
$$
The left side of this equality is reduced, while the word on the right side of this equality after cancellations starts by $z^{-2} y z^{-\alpha}$. This equality cannot hold in $F(x,y,z)$, so, this case is impossible.

Thus, we have proven that a word $w$ defines a rack $(G, *_w)$ on arbitrary group $G$ if and only if $w(x,y) = y^{-n} x y^n$ for some $n \in \mathbb{Z}$ or $w(x,y) = y x^{-1} y$. In both cases the rack $(G, *_w)$ is a quandle.
\end{proof}

\subsection{Biracks and biquandles}

If $Q$ is a verbal quandle, then it is clear that $\mathcal{B}(Q)$ is a verbal biquandle. The Wada biquandle (Example~\ref{wadaex}) is a verbal biquandle which cannot be written as $\mathcal{B}(Q)$ for any quandle $Q$. In this section we find all possible words $u,v\in F(x,y)$ such that the algebraic system $(G,\underline{*},\overline{*})$ with the operations given by $g\overline{*}h=u(g,h)$, $g\underline{*} h= v(g,h)$ for $g,h\in G$ is a birack (respectively, biquandle) for all groups $G$. 

\begin{theorem}\label{ver-biquandle-thm}
Let $u=u(x,y)$, $v=v(x,y)$ be elements from $F(x,y)$. Then the algebraic system $(G,\overline{*},\underline{*})$, where $G$ is a group and $g\overline{*}h=u(g,h)$, $g\underline{*} h= v(g,h)$ for $g,h\in G$ is a birack for every group $G$ if and only if the elements $u,v$ have one of the following forms:
\begin{align}
&\notag(1)~~u(x,y)=x, &&v(x,y) = y^{\gamma} x y^{-\gamma},&&\text{where}~\gamma \in \mathbb{Z};&&~&&~&&~\\
&\notag(2)~~u(x,y) = y^{\alpha} x y^{-\alpha}, &&v(x,y) =  x, &&\text{where}~\alpha \in \mathbb{Z};&&~&&~&&~\\
&\notag(3)~~u(x,y) = y^{-1} x y^{-1}, &&v(x,y) =  x^{-1};&&~&&~&&~&&~\\
&\notag(4)~~u(x,y) = y x^{-1}  y, &&v(x,y) =  x;&&~&&~&&~&&~\\
&\notag(5)~~u(x,y) = x y^{-2}, &&v(x,y) =  y x^{-1} y^{-1};&&~&&~&&~&&~\\
&\notag(6)~~u(x,y) = y^{-2}x, &&v(x,y) =  y^{-1} x^{-1} y;&&~&&~&&~&&~\\
&\notag(7)~~u(x,y) = x, &&v(x,y) =  y x^{-1} y;&&~&&~&&~&&~\\
&\notag(8)~~u(x,y) = x^{-1}, &&v(x,y) =  y^{-1} x^{-1} y^{-1}.&&~&&~&&~&&~
\end{align}
In particular, every such birack is a biquandle.
\end{theorem}
\begin{proof}The fact that the operations $g\overline{*}h=u(g,h)$, $g\underline{*} h= v(g,h)$ given in the formulation of the theorem define biquandles can be proved using a direct check. Let us prove that these operations are the only possible ones.

Let the algebraic system $(G,\overline{*},\underline{*})$, where $G$ is a group and $g\overline{*}h=u(g,h)$, $g\underline{*} h= v(g,h)$ for $g,h\in G$ be a birack for all groups $G$. Since the maps $\alpha_y:x\mapsto x \underline{*} y=v(x,y)$, $\beta_y:x\mapsto x \overline{*} y=u(x,y)$ are invertible, similar to Proposition~\ref{p7.2}, we conclude that the words $u(x,y)$, $v(x,y)$ have the following forms
\begin{align}
\label{simpleformbiqverb} u(x,y) = x^{\alpha} y^{\varepsilon} x^{\beta},&&v(x,y) = y^{\gamma} x^{\mu} y^{\delta}
\end{align}
for some $\alpha, \beta, \gamma, \delta \in \mathbb{Z}$, $\varepsilon, \mu \in \{\pm 1\}$. Rewriting the identities
\begin{enumerate}
\item $(x\underline{*}y)\underline{*}(z\underline{*}y)=(x\underline{*}z)\underline{*}(y\overline{*}z)$,
\item $(x\underline{*}y)\overline{*}(z\underline{*}y)=(x\overline{*}z)\underline{*}(y\overline{*}z)$,
\item $(x\overline{*}y)\overline{*}(z\overline{*}y)=(x\overline{*}z)\overline{*}(y\underline{*}z)$
\end{enumerate}
of biracks using equalities (\ref{simpleformbiqverb}) and the fact that $x \overline{*} y = y^{\alpha} x^{\varepsilon} y^{\beta}$, $x \underline{*} y = y^{\gamma} x^{\mu} y^{\delta}$, we have the equalities
\begin{align}
\label{id1}(y^{\gamma} z^{\mu} y^{\delta})^{\gamma} (y^{\gamma} x^{\mu} y^{\delta})^{\mu}  (y^{\gamma} z^{\mu} y^{\delta})^{\delta}
&=(z^{\alpha} y^{\varepsilon} z^{\beta})^{\gamma} (z^{\gamma} x^{\mu} z^{\delta})^{\mu} (z^{\alpha} y^{\varepsilon} z^{\beta})^{\delta},\\
\label{id2}(y^{\gamma} z^{\mu} y^{\delta})^{\alpha} (y^{\gamma} x^{\mu} y^{\delta})^{\varepsilon}  (y^{\gamma} z^{\mu} y^{\delta})^{\beta}
&=(z^{\alpha} y^{\varepsilon} z^{\beta})^{\gamma} (z^{\alpha} x^{\varepsilon} z^{\beta})^{\mu} (z^{\alpha} y^{\varepsilon} z^{\beta})^{\delta},\\
\label{id3}(y^{\alpha} z^{\varepsilon} y^{\beta})^{\alpha} (y^{\alpha} x^{\varepsilon} y^{\beta})^{\varepsilon}  (y^{\alpha} z^{\varepsilon} y^{\beta})^{\beta}
&=(z^{\gamma} y^{\mu} z^{\delta})^{\alpha} (z^{\alpha} x^{\varepsilon} z^{\beta})^{\varepsilon} (z^{\gamma} y^{\mu} z^{\delta})^{\beta}
\end{align}
which hold in every group $G$ for all elements $x,y,z\in G$.

Let $G$ be a free abelian group with the free generators $x,y,z$. Then equalities (\ref{id1}), (\ref{id2}), (\ref{id3}) can be rewritten in the form
\begin{align}
\notag x y^{(\gamma+\delta) (\gamma+\delta+\mu)} z^{(\gamma+\delta)\mu} &= x y^{(\gamma+\delta)\varepsilon} z^{(\alpha+\beta) (\gamma+\delta+\mu)},\\
\notag x^{\varepsilon \mu} y^{(\gamma+\delta) (\alpha+\beta+\varepsilon)} z^{(\alpha+\beta)\mu} &= x^{\varepsilon \mu} y^{(\gamma+\delta)\varepsilon} z^{(\alpha+\beta) (\gamma+\delta+\mu)},\\
\notag x y^{(\alpha+\beta) (\alpha+\beta+\varepsilon)} z^{(\alpha+\beta)\varepsilon} &= x y^{(\alpha+\beta)\mu} z^{(\alpha+\beta)(\gamma+\delta+\varepsilon)},
\end{align}
which is equivalent to the system of equations
\begin{equation}\label{defsystem}
\begin{cases}(\alpha+\beta)(\gamma+\delta)=0,\\
(\gamma+\delta) (\gamma+\delta+\mu - \varepsilon)=0,\\
(\alpha+\beta) (\alpha+\beta+\varepsilon- \mu )=0.\\
\end{cases}
\end{equation}
This system is solvable in one of the following three cases:
\begin{enumerate}
\item $\gamma+\delta=0$, $\alpha+\beta=0$,
\item $\gamma+\delta=0$, $\alpha+\beta\neq0$,
\item $\gamma+\delta\neq 0$, $\alpha+\beta=0$.
\end{enumerate}
We will consider these three cases separately.

\textit{Case 1}: $\gamma+\delta=0$, $\alpha+\beta=0$. In this case, we have 
\begin{align}
x \overline{*} y = y^{\alpha} x^{\varepsilon} y^{-\alpha},&&x \underline{*} y = y^{\gamma} x^{\mu} y^{-\gamma}.
\end{align}
Putting these expressions to equality (\ref{id1}), we have the equality
\begin{equation}\label{eq10}
y^{\gamma} z^{\mu \gamma} x z^{-\mu \gamma} y^{-\gamma} = z^{\alpha} y^{\varepsilon \gamma} z^{\gamma-\alpha} x z^{-\gamma+\alpha} y^{-\varepsilon \gamma} z^{-\alpha}.
\end{equation}
This equality holds in the free group $F(x,y,z)$ with the free generators $x,y,z$ if and only if $\gamma=0$ or $\alpha=0$. 

\textit{Case 1.1}: $\alpha=0$. In this case equality (\ref{eq10}) can be rewritten in the form
$$
y^{\gamma} z^{\mu \gamma} x z^{-\mu \gamma} y^{-\gamma} =  y^{\varepsilon \gamma} z^{\gamma} x z^{-\gamma} y^{-\varepsilon \gamma}.
$$
This equality holds in the free group $F(x,y,z)$ if and only if  $\varepsilon = \mu =1$. Thus, the elements $x\overline{*}y$, $x\underline{*}y$ in this situation have the form 
\begin{align} \label{eq11}
x \overline{*} y =  x,&&x \underline{*} y = y^{\gamma} x y^{-\gamma}.
\end{align}

\textit{Case 1.2}: $\gamma = 0$. In this case equality (\ref{eq10}) can be rewritten as $x = x$, and  the elements $x\overline{*}y$, $x\underline{*}y$ in this situation have the form
\begin{align}
\notag x \overline{*} y = y^{\alpha} x^{\varepsilon} y^{-\alpha},&&x \underline{*} y = x^{\mu}.
\end{align}
Putting these elements in (\ref{id3}), we get the equality
$$
y^{\alpha} z^{\alpha \varepsilon} x z^{-\alpha \varepsilon} y^{-\alpha} =  y^{\alpha \mu} z^{\alpha} x z^{-\alpha} y^{-\alpha \mu},
$$
which holds in the free group $F(x,y,z)$ if an only if $\varepsilon = \mu =1$. 
In this situation, the elements $x\overline{*}y$, $x\underline{*}y$  have the form
\begin{align} \label{eq12}
x \overline{*} y = y^{\alpha} x y^{-\alpha},&&x \underline{*} y = x.
\end{align}
Thus, we proved that in the case when $\gamma+\delta=0$, $\alpha+\beta=0$ the elements $x\overline{*}y$, $x\underline{*}y$ are given either by (\ref{eq11}) or by (\ref{eq12}).

\textit{Case 2}: $\gamma+\delta=0$, $\alpha+\beta\neq0$. In this case from system (\ref{defsystem}) it follows that $\mu = -\varepsilon$ and $\alpha+\beta + 2 \varepsilon= 0$. The elements $x\overline{*}y$, $x\underline{*}y$ in this situation are given by the formulas
\begin{align}
x \overline{*} y = y^{\alpha} x^{\varepsilon} y^{-\alpha-2\varepsilon},&&x \underline{*} y = y^{\gamma} x^{-\varepsilon} y^{-\gamma}.
\end{align}
 Substituting these operations to (\ref{id1}), we obtain
 \begin{equation}\label{eq1}
y^{\gamma} z^{-\varepsilon \gamma} x z^{\varepsilon \gamma} y^{-\gamma} = \left( z^{\alpha} y^{\varepsilon} z^{-\alpha-2\varepsilon} \right)^{\gamma}  z^{\gamma} x  z^{-\gamma} \left( z^{\alpha} y^{\varepsilon} z^{-\alpha-2\varepsilon} \right)^{-\gamma}.
\end{equation}
Depending on $\gamma$, we have the following cases.

{\it Case 2.1}: $\gamma > 0$. In this case equality (\ref{eq1}) can be rewritten as 
$$
y^{\gamma} z^{-\varepsilon \gamma}  x z^{\varepsilon \gamma} z^{-\gamma} = \underbrace{(z^{\alpha} y^{\varepsilon} z^{-\alpha-2 \varepsilon}) \ldots (z^{\alpha} y^{\varepsilon} z^{-\alpha-2 \varepsilon})}_{\gamma~ \mbox{\tiny times}} z^{\gamma}  x z^{-\gamma}
\underbrace{(z^{\alpha+2 \varepsilon} y^{-\varepsilon} z^{-\alpha}) \ldots (z^{\alpha+2 \varepsilon} y^{-\varepsilon} z^{-\alpha})}_{\gamma~ \mbox{\tiny times}}.
$$
Looking at this equality in the free group $F(x,y,z)$, we conclude that $\alpha=0$. Comparing  syllable lengths of the left and the right sides of the preceding equality, we see that $\gamma=1$. Thus, we can rewrite this equality in the form
$$
y z^{-\varepsilon}  x z^{\varepsilon} y^{-1} = y^{\varepsilon} z^{-2\varepsilon+1}  x z^{2\varepsilon-1} y^{-\varepsilon},
$$
which is true in $F(x,y,z)$ if and only if $\varepsilon = 1$. Hence, we get operations
\begin{align}
\label{case21wad}x \overline{*} y = x y^{-2},&&x \underline{*} y = y x^{-1} y^{-1}.
\end{align}

{\it Case 2.2}: $\gamma = 0$. In this case equality (\ref{eq1}) obviously holds, and we have the operations
\begin{align}
\label{case22wad}
x \overline{*} y = y^{\alpha} x^{\varepsilon} y^{-\alpha-2\varepsilon},&&x \underline{*} y = x^{-\varepsilon}.
\end{align}
Substituting these operations to equality (\ref{id2}) we get the equality
\begin{equation}\label{subst2}
z^{-\varepsilon \alpha}  x^{-1} z^{\varepsilon (\alpha+2\varepsilon)} = \left( z^{\alpha}  x^{\varepsilon} z^{-\alpha-2\varepsilon} \right)^{-\varepsilon}.
\end{equation}
If $\varepsilon=1$, then equality (\ref{subst2}) can be rewritten as
$$
z^{-\alpha}  x^{-1} z^{\alpha+2} = z^{\alpha+2}  x^{-1} z^{-\alpha}
$$
and it holds if and only if $\alpha = -1$. Hence, we get operations
\begin{align}\label{op1}
x \overline{*} y = y^{-1} x y^{-1},&&x \underline{*} y =  x^{-1}.
\end{align}
If $\varepsilon=-1$, then equality (\ref{subst2}) can be rewritten as the equality
$$
z^{\alpha}  x^{-1} z^{2-\alpha} = z^{\alpha}  x^{-1} z^{-\alpha+2}
$$
which always holds. Hence, we get operations
\begin{align}\notag
x \overline{*} y = y^{\alpha} x^{-1} y^{2-\alpha},&&x \underline{*} y =  x.
\end{align}
Putting these operations to equality (\ref{id3}), we get
$$
\left( y^{\alpha}  z^{-1} y^{2-\alpha} \right)^{\alpha} y^{\alpha-2}  x y^{-\alpha} \left( y^{\alpha}  z^{-1} y^{2-\alpha} \right)^{2-\alpha} =
y^{\alpha}  z^{\alpha-2} x z^{-\alpha} y^{2-\alpha}.
$$
Comparing   syllable lengths of the left and the right sides of this equality, we see that $\alpha=1$. Hence, we get operations
\begin{align}\label{op2.1}
x \overline{*} y =  y x^{-1} y,&&x \underline{*} y =  x.
\end{align}

{\it Case 2.3}: $\gamma < 0$.  In this case equality (\ref{eq1}) can be rewritten in the form
$$
y^{\gamma} z^{-\varepsilon \gamma}  x z^{\varepsilon \gamma} y^{-\gamma} = \underbrace{(z^{\alpha+2\varepsilon} y^{-\varepsilon} z^{-\alpha}) \ldots (z^{\alpha+2\varepsilon} y^{-\varepsilon} z^{-\alpha})}_{-\gamma~ \mbox{\tiny times}} z^{\gamma}  x z^{-\gamma}
\underbrace{(z^{\alpha} y^{\varepsilon} z^{-\alpha-2 \varepsilon}) \ldots (z^{\alpha} y^{\varepsilon} z^{-\alpha-2 \varepsilon})}_{-\gamma~ \mbox{\tiny times}}.
$$
Looking at this equality in the free group $F(x,y,z)$, we conclude that $\alpha+2\varepsilon = 0$. Comparing  syllable lengths of the left and the right sides of the preceding equality, we see that $\gamma=-1$. Thus, we can rewrite this equality in the form
$$
y^{-1} z^{\varepsilon}  x z^{-\varepsilon} y = y^{-\varepsilon}  z^{2\varepsilon-1}  x z^{1-2\varepsilon}
 y^{\varepsilon},
$$
which is solvable in $F(x,y,z)$ if and only if $\varepsilon=1$. Hence, we get operations
\begin{align}\label{case23wad}
x \overline{*} y = y^{-2} x,&&x \underline{*} y = y^{-1} x^{-1} y.
\end{align}
Thus, we proved that in the case when $\gamma+\delta=0$, $\alpha+\beta\neq0$ the elements $x\overline{*}y$, $x\underline{*}y$ are given by one of (\ref{case21wad}), (\ref{op1}), (\ref{op2.1}), (\ref{case23wad}).

\textit{Case 3}: $\gamma+\delta\neq0$, $\alpha+\beta=0$. In this case from system (\ref{defsystem}) follows that $\mu = -\varepsilon$ and $\delta+\gamma - 2 \varepsilon= 0$. The elements $x\overline{*}y$, $x\underline{*}y$ in this situation are given by the formulas
\begin{align}\notag
x \overline{*} y = y^{\alpha} x^{\varepsilon} y^{-\alpha},&&x \underline{*} y = y^{\gamma} x^{-\varepsilon} y^{2\varepsilon-\gamma}.
\end{align}
Putting these operations to (\ref{id1}) we get the equality
\begin{equation}\label{eq3}
\left( y^{\gamma} z^{-\varepsilon} y^{2\varepsilon-\gamma} \right)^{\gamma}
\left( y^{\gamma-2\varepsilon} x^{\varepsilon} y^{-\gamma} \right)^{\varepsilon}
\left( y^{\gamma} z^{-\varepsilon} y^{2\varepsilon-\gamma} \right)^{2\varepsilon-\gamma}
= z^{\alpha} y^{\varepsilon \gamma} z^{-\alpha}
\left( z^{\gamma-2\varepsilon} x^{\varepsilon} z^{-\gamma} \right)^{\varepsilon}
z^{\alpha} y^{2-\varepsilon\gamma} z^{-\alpha}.
\end{equation}
Depending on $\varepsilon$, we have two cases.

\textit{Case 3.1}: $\varepsilon = 1$. In this case equality (\ref{eq3}) has the form
$$
\left( y^{\gamma} z^{-1} y^{2-\gamma} \right)^{\gamma}
 y^{\gamma-2} x y^{-\gamma}
\left( y^{\gamma} z^{-1} y^{2-\gamma} \right)^{2-\gamma}
= z^{\alpha} y^{\gamma} z^{-\alpha}
z^{\gamma-2} x z^{-\gamma} z^{\alpha} y^{2-\gamma} z^{-\alpha}.
$$
Comparing   syllable lengths of the left and the right sides of this equality, we see that $\gamma=1$, and the equality can be rewritten in the form
$$
y z^{-1} x z^{-1} y = z^{\alpha} y z^{-\alpha} z^{-1} x z^{-1} z^{\alpha} y z^{-\alpha}.
$$
This equality holds in $F(x,y,z)$ if and only if $\alpha=0$. Hence, we get operations
\begin{align}\label{case3last}
x \overline{*} y =  x,&&x \underline{*} y = y x^{-1} y.
\end{align}

{\it Case 3.2}: $\varepsilon = -1$. In this case equality (\ref{eq3}) has the form
$$
\left( y^{\gamma} z y^{-2-\gamma} \right)^{\gamma}
 y^{\gamma} x y^{-2-\gamma}
\left( y^{\gamma} z y^{-2-\gamma} \right)^{-2-\gamma}
= z^{\alpha} y^{-\gamma} z^{-\alpha}
z^{\gamma} x z^{-2-\gamma} z^{\alpha} y^{2+\gamma} z^{-\alpha}.
$$
Comparing   syllable lengths of the left and the right sides of this equality, we see that $\gamma=-1$, and the equality can be rewritten as
$$
y z^{-1} x z^{-1} y = z^{\alpha} y z^{-\alpha-1}  x  z^{\alpha-1} y z^{-\alpha}.
$$
This equality holds in $F(x,y,z)$ if and only if $\alpha=0$. Hence, we get operations
\begin{align}\label{case32last}
x \overline{*} y =  x^{-1},&&x \underline{*} y = y^{-1} x y^{-1}.
\end{align}
Thus, we proved that in the case when $\gamma+\delta\neq0$, $\alpha+\beta=0$ the elements $x\overline{*}y$, $x\underline{*}y$ are given either by (\ref{case3last}) or by (\ref{case32last}).

Summarizing Case 1, Case 2 and Case 3 together, we conclude that the elements $x\overline{*}y=u(x,y)$, $x\underline{*}y=v(x,y)$ are defined by one of (\ref{eq11}), (\ref{eq12}), (\ref{case21wad}), (\ref{op1}), (\ref{op2.1}), (\ref{case23wad}), (\ref{case3last}), (\ref{case32last}).
These are exactly the $8$ cases enumerated in the formulation of the theorem.
\end{proof}

\section{Constructions of biquandles from quandles}\label{newconstructionsnew}

Let $B=(X,\underline{*},\overline{*})$ be a biquandle. For elements $x,y\in B$ define the element $x*y$ by the rule
$$x*y=(x\underline{*}y)\overline{*}^{-1}y.$$
The set $X$ with the operation $*$ is a quandle \cite{Ashihara} (see also \cite{Horvat}), which is denoted by $\mathcal{Q}(B)=(X,*)$ and called the associated quandle of $B$. It can be seen that $\mathcal{Q}: \textbf{Bqnd} \to \textbf{Qnd}$ given by $B\mapsto\mathcal{Q}(B)$ is a functor from the category of biquandles to the category of quandles \cite[Lemma 3.1]{Horvat}. If $Q$ is a quandle, then it is clear that $\mathcal{Q}(\mathcal{B}(Q))=Q$. However, $\mathcal{B}(\mathcal{Q}(B))\not\simeq B$ in general. Note that the functors $\mathcal{Q}: \textbf{\text{Bqnd}}\to \textbf{\text{Qnd}}$, $\mathcal{B}:\textbf{\text{Qnd}}\to\textbf{\text{Bqnd}}$ are not adjoint to each other.
\begin{problem}What is the adjoint to the functor $\mathcal{Q}$?
\end{problem}

Let $Q=(X,*)$ be a quandle. According to \cite{Horvat} a \textit{biquandle structure} on  $Q$ is a family of automorphisms $\{\beta_y:Q\to Q~|~y\in X\}\subset {\rm Aut}(Q)$ which satisfies the following two conditions:
\begin{enumerate}
\item $\beta_{\beta_y(x*y)}\beta_y=\beta_{\beta_x(y)}\beta_x$ for all $x,y\in X$,
\item the map $y\mapsto \beta_y(y)$ is a bijection of $X$. 
\end{enumerate}
\begin{example}Let $f$ be an automorphism of a quandle $Q$. Then the set  $\{\beta_y=f~|~y \in X\}$ obviously forms a biquandle structure on $Q$, which is called the \textit{constant biquandle structure}.
\end{example}
\begin{example} Let $Q=(X, *)$ be a quandle. For $x\in X$ denote by $\beta_x=S_x^{-1}$. Then the family of automorphisms $\{\beta_x~|~x \in X \}$ forms a non-constant biquandle structure on $Q$. 
\end{example}
\begin{example}
Let $A= \mathbb{Z}_3^n$ for $n \geq1$, and ${\rm T}(A)$ be the Takasaki quandle on $A$. Then the family of automorphisms $\{\beta_x~|~x \in A \}$, where $\beta_x(y)= y+x$ for $x,y\in A$, forms a non-constant biquandle structure on ${\rm T}(A)$.
\end{example}
The following result is proved in \cite[Theorem 3.2]{Horvat}.
\begin{theorem}\label{horvatMain}
Let $\{\beta_y:Q\to Q~|~y\in X\}$ be a biquandle structure on a quandle  $Q=(X,*)$. Define two binary operations on $X$ by the rules $x\underline{*}y=\beta_y(x*y)$, $x\overline{*}y=\beta_y(x)$. Then $B=(X,\underline{*},\overline{*})$  is a biquandle with $\mathcal{Q}(B)=Q$. Moreover, every biquandle can be constructed in this way. 
\end{theorem}
We say that a biquandle $B$ constructed in Theorem~\ref{horvatMain} is an associated to a biquandle structure $A=\{\beta_y:Q\to Q~|~y\in X\}$ on a quandle $Q$.

Thus, by Theorem~\ref{horvatMain} biquandles can be constructed from quandles, and it would be useful 
to have some canonical ways of constructing biquandles from quandles. In this section, we introduce new general constructions of biquandles from quandles: in Section~\ref{sunion} we introduce biquandles on unions of quandles, in Section~\ref{newproductofquandles} we introduce biquandles on products of quandles, and in Section~\ref{liftingconstruction} we study conditions when for a surjective quandle homomorphism $p:\widetilde{Q}\to Q$ biquandle structures on $Q$ can be lifted to biquandle structures on $\widetilde{Q}$.

\subsection{Biquandles on unions of quandles}\label{sunion}
The following method of constructing a new quandle from given quandles is introduced in \cite[Proposition 9.2]{BNS}.
\begin{proposition}\label{oldunion}
Let $Q_1=(X_1,*_1)$, $Q_2=(X_2,*_2)$ be quandles, and $\sigma: Q_1 \to  {\rm Conj}_{-1} \left({\rm Aut}(Q_2) \right)$, $\tau: Q_2 \to  {\rm Conj}_{-1} \left({\rm Aut}(Q_1) \right)$ be quandle homomorphisms. Then the set $X_1 \sqcup X_2$ with the operation
$$
x* y=\begin{cases}
x*_1y,& x, y \in Q_1, \\
x*_2 y,  &x, y \in Q_2, \\
{\tau(y)}(x),  &x \in Q_1, y \in Q_2, \\
{\sigma(y)}(x) &x \in Q_2, y \in Q_1
\end{cases}
$$
is a quandle if and only if the following conditions hold:
\begin{enumerate}
\item $\tau(z)(x)*_1 y=\tau\left(\sigma(y)(z)\right)(x*_1 y)$ for $x, y \in Q_1$ and $z \in Q_2$,
\item $\sigma(z)(x)*_2 y=\sigma\left(\tau(y)(z)\right)(x*_2 y)$ for $x, y \in Q_2$ and $z \in Q_1$.
\end{enumerate}
\end{proposition}
\noindent The quandle $(X_1\sqcup X_2,*)$ introduced in Proposition~\ref{oldunion} is denoted by $Q=Q_1\underset{\sigma,\tau}{\sqcup} Q_2$ and is called the \textit{union of quandles $Q_1,Q_2$ with respect to $\sigma,\tau$}. See, for example, \cite[Theorem 3.3]{BarNas} for applications of union quandles. If $\sigma$, $\tau$ are trivial homomorphism, i.~e. $\sigma(Q_1)=\{id\}$, $\tau(Q_2)=\{id\}$, then the quandle $Q_1\underset{\sigma,\tau}{\sqcup} Q_2$ is denoted by $Q_1\sqcup Q_2$ and is called the \textit{union of quandles $Q_1,Q_2$.} In this section we introduce a general way of constructing biquandles on a union of two quandles.

Let $Q_1=(X_1,*_1)$, $Q_2=(X_2,*_2)$ be quandles, and $Q=Q_1\sqcup Q_2=(X_1\sqcup X_2, *)$ be the union of quandles $Q_1$, $Q_2$. If $f$ is an automorphism of $Q_i$, then $f$ can be extended to an automorphism of $Q$ in the following way
\begin{equation}\label{tildedefinition}
\widetilde{f}(x)=\begin{cases}
f(x),& x \in X_i\\
x, & x\not\in X_i.\\
\end{cases}
\end{equation} 
Let $\phi: Q_1 \to {\rm Conj}_{-1} ({\rm Aut}(Q_2))$, $\psi: Q_2 \to {\rm Conj}_{-1} ({\rm Aut} (Q_1))$ be homomorphisms. For $x\in Q_1$, $y\in Q_2$ denote by $\phi(x)=\phi_x$, $\psi(y)=\psi_y$. Further, for  $a\in X_1\sqcup X_2$ denote by $\beta_a$ the following map from $X_1\sqcup X_2$ to itself 
\begin{equation}\label{betaunion}
\beta_a=\begin{cases}
\widetilde{\phi_a},& a \in X_1,\\
\widetilde{\psi_a},& a \in X_2.
\end{cases}
\end{equation}
It is clear that $\beta_a$ is an automorphism of $Q_1\sqcup Q_2$. The following result gives conditions under which the set $\{\beta_a~|~a\in X_1\sqcup X_2\}$ forms a biquandle structure on $Q_1\sqcup Q_2$.

\begin{theorem}\label{gen-biquandle-union}
Let $Q_1=(X_1,*_1)$, $Q_2=(X_2,*_2)$ be quandles, $X=X_1\sqcup X_2$, and $Q= (X,*)$ be the union of quandles $Q_1$, $Q_2$. Let $\phi: Q_1 \to {\rm Conj}_{-1} ({\rm Aut}(Q_2))$ and $\psi: Q_2 \to {\rm Conj}_{-1} (\Aut (Q_1))$ be quandle homomorphisms such that 
\begin{align}
\label{unionconditions}\phi_{x_1}= \phi_{\psi_{x_2}(x_1)},&&\psi_{x_2}= \psi_{\phi_{x_1}(x_2)}
\end{align}
for all $x_1 \in Q_1$, $x_2 \in Q_2$. Then the family of automorphisms $\{\beta_y~|~y \in X \}$ defined by (\ref{betaunion}) is a biquandle structure on $Q$.
\end{theorem}

\begin{proof}
Since $\phi,\psi$ are quandle homomorphisms, for all $x_1,y_1\in X_1$, $x_2,y_2\in Q_2$, we have
\begin{align}
\notag \phi(x_1*y_1)=\phi(y_1)\phi(x_1) \phi(y_1)^{-1},&&\psi(x_2*y_2)=\psi(y_2)\psi(x_2) \psi(y_2)^{-1}.
\end{align} 
From these equalities and definition (\ref{betaunion}) follows that the equality
\begin{equation}\label{beta-conjugation}
\beta_{x*y}=\beta_y\beta_x \beta_y^{-1}
\end{equation}
holds for all $x,y\in X$.

Let $y$ be an arbitrary element of $Q$. If $y\in X_1$, then by (\ref{betaunion}) we have $\beta_y(y)=\widetilde{\phi_y}(y)$. By (\ref{tildedefinition}) we have $\widetilde{\phi_y}(y)=y$, and therefore $\beta_y(y)=y$. Similarly, if $y\in X_2$, then $\beta_y(y)=\widetilde{\psi_y}(y)=y$. Hence, the map $y\mapsto\beta_y(y)$ is identity, and hence is a bijection. So, in order to prove that $\{\beta_y~|~y\in Q\}$ is a biquandle structure on $Q$ we need to check that the equality 
\begin{equation}\label{necessaryunion}
\beta_{\beta_y(x*y)}\beta_y(z)=\beta_{\beta_x(y)}\beta_x(z)
\end{equation}
holds for all $x,y,z\in Q$. Depending on $x,y,z$, we have the following cases.

\textit{Case 1}: $x, y, z \in X_1$. In this case
$$ 
\beta_{\beta_y(x*y)}\beta_y(z) = \beta_{x*y} (z)= z=\beta_{x}(z)= \beta_{y}\beta_x(z)=\beta_{\beta_x(y)}\beta_x(z),$$
i.~e. equality (\ref{necessaryunion}) holds. The case when  $x, y, z \in X_2$ is similar.

\textit{Case 2}: $x, y \in X_1$, $z \in X_2$. In this case
\begin{align}
\notag\beta_{\beta_y(x*y)}\beta_y(z) &= \beta_{x*y} \beta_y(z)\\
\notag&= \beta_{y} \beta_x(z),~ \textrm{by}~\eqref{beta-conjugation}\\
\notag&= \beta_{\beta_x(y)} \beta_x(z),
\end{align}
i.~e. equality (\ref{necessaryunion}) holds. The case when  $x, y \in X_2$ and $z \in X_1$ is similar.

\textit{Case 3}: $x, z \in X_1$, $y \in X_2$. In this case 
\begin{align}
\notag\beta_{\beta_y(x*y)}\beta_y(z) &=\beta_{\beta_y(x)}\beta_y(z)\\
\notag&= \beta_y(z)\\
\notag&= \beta_{\beta_x(y)} (z),~ \textrm{by}~(\ref{unionconditions})\\
\notag&= \beta_{\beta_x(y)} \beta_x(z),
\end{align}
i.~e. equality (\ref{necessaryunion}) holds. The case when  $x\in X_2$ and $y \in X_1$ is similar.
\end{proof}
Theorem~\ref{gen-biquandle-union} gives a biquandle structure on the union $Q_1\sqcup Q_2$ of quandles $Q_1,Q_2$. By Theorem~\ref{horvatMain} using a biquandle structure on a given quandle one can construct a biquandle. The biquandle associated to the biquandle structure  $\{\beta_a~|~a\in X_1\sqcup X_2\}$ of the quandle $(X_1,*_1)\sqcup (X_2,*_2)$ is a biquandle on the set $X_1\sqcup X_2$ with operations given by
\begin{align}
\notag a,b\in X_1 &\Rightarrow a\overline{*}b=a, a\underline{*}b=a*_1b,\\
\notag a,b\in X_2 &\Rightarrow a\overline{*}b=a, a\underline{*}b=a*_2b,\\
\notag a\in X_1, b\in X_2 &\Rightarrow a\overline{*}b=\psi_b(a), a\underline{*}b=\psi_b(a),\\
\notag a\in X_2, b\in X_1 &\Rightarrow a\overline{*}b=\phi_b(a), a\underline{*}b=\phi_b(a).
\end{align} 
This biquandle is denote by $B(Q_1~{}_{\varphi}\bigsqcup{}_{\psi}~Q_2)$, and is called the \textit{union} biquandle.
\begin{proposition}
Let $Q_1=(X_1,*_1)$, $Q_2=(X_2,*_2)$ be involutory quandles, and $Q=Q_1\sqcup Q_2$. Let $\phi: Q_1 \to {\rm Conj}_{-1} ({\rm Aut}(Q_2))$, $\psi: Q_2 \to {\rm Conj}_{-1} ({\rm Aut} (Q_1))$ be homomorphisms which satisfy (\ref{unionconditions}). If $\phi_{x}^2=id$ for all $x\in Q_1$, $\psi_y^2=id$ for all $y\in Q_2$, then the biquandle $B(Q_1~{}_{\varphi}\bigsqcup{}_{\psi}~Q_2)$ is involutory.
\end{proposition}
\begin{proof} In order to prove that $B(Q_1~{}_{\varphi}\bigsqcup{}_{\psi}~Q_2)$ is involutory, we need to check that equalities (\ref{inv-biquandle-condition}) hold for all $x,y\in X_1\sqcup X_2$. Depending on $x,y$, we have the following cases.  

\textit{Case 1}: $x, y \in X_1$. In this case
\begin{align}
\notag x \underline{*} (y \overline{*} x) &= x*_1y=x \underline{*} y,\\
\notag x \overline{*} (y \underline{*} x) &=  x=  x \overline{*} y,\\
\notag (x \underline{*} y) \underline{*} y &=  (x*_1y) *_1 y=x,~\text{since}~ Q_1~\textrm{is involutory},\\
\notag(x \overline{*} y) \overline{*} y&=(x*_1y)*_1y = x,~\text{since}~ Q_1~\textrm{is involutory},
\end{align}
i.~e. equalities (\ref{inv-biquandle-condition}) hold. The case when $x,y\in X_2$ is similar.

\textit{Case 2}: $x \in X_1$, $y \in X_2$. In this case
$$
\notag x \underline{*} (y \overline{*} x) = x \underline{*} \phi_x(y)= \psi_{\phi_x(y)}(x)= \psi_y(x)= x \underline{*} y,
$$
i.~e. the first equality from  (\ref{inv-biquandle-condition}) holds. 
Similarly, we have
\begin{align}
\notag x \overline{*} (y \underline{*} x) &= x \overline{*} \phi_x(y)= \psi_{\phi_x(y)}(x)=\psi_y(x)=x \overline{*} y,\\
\notag(x \underline{*} y) \underline{*} y &= \psi_y(x)\underline{*} y= \psi_y^2(x)=x,~\text{since}~\psi_y^2=id,\\
\notag(x \overline{*} y) \overline{*} y &= \psi_y(x)\overline{*} y= \psi_y^2(x)=x,~\text{since}~\psi_y^2=id,
\end{align}
i.~e. all equalities  (\ref{inv-biquandle-condition}) hold. The case when $x\in X_2$, $y\in X_1$ is similar. Thus, we proved that  $B(Q_1~{}_{\varphi}\bigsqcup{}_{\psi}~Q_2)$ is involutory.
\end{proof}
If $\phi: Q_1 \to {\rm Conj}_{-1} ({\rm Aut}(Q_2))$, $\psi:Q_2 \to {\rm Conj}_{-1} ({\rm Aut}(Q_1))$ are constant maps, i.~e. for all $x\in Q_1$, $y\in Q_2$ we have $\phi_x=g$, $\psi_y=f$ for some fixed automorphisms $f\in {\rm Aut}(Q_2)$, $g\in {\rm Aut}(Q_1)$, then conditions (\ref{unionconditions}) obviously hold. Hence, Theorem~\ref{gen-biquandle-union} implies the following result. 
\begin{corollary}\label{constant-union-biquandle}
Let $Q_1=(X_1,*_1)$, $Q_2=(X_2,*_2)$ be quandles, $f \in {\rm Aut}(Q_1)$ and $g \in {\rm Aut}(Q_2)$. For each $y \in X_1\sqcup X_2$, define 
$$\beta_y=\begin{cases}
\tilde{g},& y \in X_1,\\
\tilde{f},& y \in X_2.
\end{cases}$$ 
Then the family of automorphisms $\{\beta_y~|~y \in X_1\sqcup X_2 \}$ is a biquandle structure on $Q= Q_1 \sqcup Q_2$.
\end{corollary}
The biquandle structure described in Corollary~\ref{constant-union-biquandle} is the next simplest biquandle structure after the constant biquandle structures since it contains only two automorphisms. The biquandle associated to the biquandle structure $\{\beta_a~|~a\in X_1\sqcup X_2\}$ described in Corollary~\ref{constant-union-biquandle} is a biquandle on the set $X_1\sqcup X_2$ with the operations given by
\begin{align}
\label{rulesunion1} a,b\in X_1 &\Rightarrow a\overline{*}b=a, a\underline{*}b=a*_1b,\\
\label{rulesunion2} a,b\in X_2 &\Rightarrow a\overline{*}b=a, a\underline{*}b=a*_2b,\\
\label{rulesunion3} a\in X_1, b\in X_2 &\Rightarrow a\overline{*}b=f(a), a\underline{*}b=f(a),\\
\label{rulesunion4} a\in X_2, b\in X_1 &\Rightarrow a\overline{*}b=g(a), a\underline{*}b=g(a).
\end{align} 
This biquandle is denoted by $B(Q_1~{}_{g}\bigsqcup{}_{f}~Q_2)$. 

\begin{remark}\label{uniontrivial}The trivial quandle $T_{n}$ can be written as a union $T_{n}=T_m\sqcup T_k$ of trivial quandles $T_m$, $T_k$, where $m+k=n$. Every permutation of the trivial quandle is an automorphism of the trivial quandle, so, Corollary~\ref{constant-union-biquandle} gives $m!k!$ biquandle structures on $T_{n}$ for each decomposition $T_n=T_m\sqcup T_k$.
\end{remark}

\begin{example}\label{trivial-biquandle-distinguish}
 As we noticed at the end of Section~\ref{quandlesprem}, quandle-coloring invariant defined  by the trivial quandle $T_n$ cannot distinguish a $r$-component virtual link from the trivial link with $r$ components. However, if  from the trivial quandle $T_n$ we construct a biquandle $B$ using some biquandle structure, then the biquandle-coloring invariant defined by $B$ would be able to distinguish different virtual links with the same number of components. For example, let $m,k\geq 2$ be integers $Q_1=T_m=\{x_1,\dots,x_m\}$, $Q_2=T_k=\{y_1,\dots,y_k\}$ be trivial quandles, $f=(x_1,x_2,\dots,x_m)$ be a cycle on $Q_1$, and $g=(y_1,y_2,\dots,y_k)$ be a cycle on $Q_2$. It is clear that $f$ is an automorphism of $Q_1$, and $g$ is an automorphism of $Q_2$. Thus, using formulas (\ref{rulesunion1}), (\ref{rulesunion2}), (\ref{rulesunion3}), (\ref{rulesunion4}), we can define a biquandle $B=B(Q_1~{}_{g}\bigsqcup{}_{f}~Q_2)$ on the set $\{x_1,\dots,x_m,y_1,\dots,y_k\}$.

Let us prove the biquandle-coloring invariant defined by $B$ distinguishes the virtual Hopf link $H$ (see Figure \ref{biquandlerelations}) from the trivial $2$-component link $U$. 
\begin{figure}[hbt!]
\noindent\centering{
\includegraphics[height=30mm]{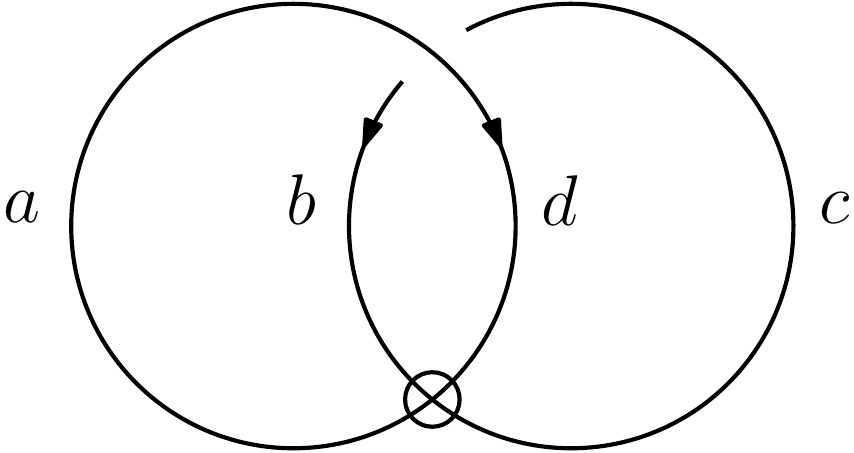}}
\caption{Labels of arcs of the virtual Hopf link.}
\label{biquandlerelations}
\end{figure}

Since $U$ has no crossings, every coloring of $U$ is proper, and therefore $C_B(U)=(m+k)^2$. In order to calculate $C_B(H)$, label the arcs of $H$ by  elements $a,b,c,d\in B$ as it is depicted on  
Figure~\ref{biquandlerelations}. The coloring by $a,b,c,d$ is proper if and only if the following equalities hold. 
\begin{align}
\label{coloring1} &d=a,&&c=b,&&\text{from the virtual crossing},\\
\label{coloring2} &c=b\underline{*}a,&&d=a\overline{*}b,&&\text{from the classical crossing}.
\end{align}
From equality (\ref{coloring1}) it is clear that the proper coloring of $H$ is completely defined by labels $a,b$. If $a,b\in Q_1$ or $a,b\in Q_2$, then equalities (\ref{coloring1}), (\ref{coloring2}) obviously hold, and therefore every coloring of $H$ with $a,b\in Q_1$ or $a,b\in Q_2$ is proper. There are $m^2+k^2$ such colorings ($m^2$ colorings when $a,b\in Q_1$, and $k^2$ colorings when $a,b\in Q_2$). If $a\in Q_1$, $b\in Q_2$, then $b\underline{*}a=g(b)\neq b$ (since $k\geq 2$, and $g$ is the cycle $g=(y_1,y_2,\dots,y_k)$). Hence, equalities (\ref{coloring1}), (\ref{coloring2}) cannot hold, and the coloring of $H$ with $a\in Q_1$, $b\in Q_2$ is not proper. Similarly, we can prove that the coloring of $H$ with $a\in Q_2$, $b\in Q_1$ is not proper, and hence $C_B(H)=m^2+k^2\neq (m+k)^2=C_B(U)$. Thus, biquandle-coloring invariant defined by $B$ distinguishes the virtual Hopf link from the trivial link with $2$ components, while the quandle-coloring invariant defined by $\mathcal{Q}(B)=T_{n+m}$ does not distinguish $H$ from $U$.
\end{example}

The preceding example shows that biquandles constructed on trivial quandles can distinguish different virtual knots and links (despite the fact that the trivial quandle cannot distinguish any two links with the same number of components). The following problem appears naturally in this context.
 
\begin{problem}\label{classificationTrivial}
Classify biquandle structures on the trivial quandle $T_n$ with $n$ elements.
\end{problem}

Due to Remark~\ref{uniontrivial}, using Theorem~\ref{gen-biquandle-union} and Corollary~\ref{constant-union-biquandle} it is possible to construct a lot of non-trivial biquandle structures on trivial quandles. 

\subsection{Biquandles on products of quandles}\label{newproductofquandles} 

In \cite{Kamada} the following construction of a biquandle from two quandles was introduced (see also \cite[Section~5]{Horvat}). Let $Q_1=(X_1,*_1)$, $Q_2=(X_2,*_2)$ be quandles. For $(x,a), (y,b)\in X_1\times X_2$ denote by
\begin{align}
(x,a)\underline{*}(y,b)=(x*_1y,a),&&(x,a)\overline{*}(y,b)=(x,a*_2b).
\end{align}
Then $(X_1\times X_2,\underline{*},\overline{*})$ is a biquandle called the \textit{product biquandle} of quandles $Q_1$, $Q_2$. In this section we generalize this construction.

Let $Q_1=(X_1, *_1)$, $Q_2=(X_2, *_2)$ be quandles, and $\phi: Q_1 \to {\rm Conj}_{-1} ({\rm Aut}(Q_2))$, $\psi: Q_2 \to {\rm Conj}_{-1} ({\rm Aut} (Q_1))$ be homomorphisms with $\phi(x)=\phi_x$ for $x\in Q_1$, $\psi(y)=\psi_y$ for $y\in Q_2$. For $(x, a), (y, b) \in X_1 \times X_2$ denote by $(x, a) \underline{*} (y, b)$, $(x, a) \overline{*} (y, b)$ the following elements of $X_1\times X_2$.
\begin{align}
\label{productrule1}(x, a) \underline{*} (y, b)= \big(\psi_b(x*_1y), \phi_y(a) \big)\\
\label{productrule2}(x, a) \overline{*} (y, b)= \big(\psi_b(x), \phi_y(a *_2 b)\big)
\end{align}
We have the following result.

\begin{theorem}\label{action-product-biquandle}
Let $Q_1=(X_1, *_1)$, $Q_2=(X_2, *_2)$ be quandles, and $\phi: Q_1 \to {\rm Conj}_{-1} ({\rm Aut}(Q_2))$, $\psi: Q_2 \to {\rm Conj}_{-1} ({\rm Aut} (Q_1))$ be homomorphisms with $\phi(x)=\phi_x$ for $x\in Q_1$, $\psi(y)=\psi_y$ for $y\in Q_2$.
\begin{enumerate}
\item If $\psi(Q_2)=\{f\}$ and $\phi_{f(x *_1y)}=\phi_{f(y)} \phi_x \phi_y^{-1}$ for all $x, y \in X_1$, then $(X_1 \times X_2, \underline{*}, \overline{*})$ is a biquandle.
\item If $\phi(Q_1)=\{g\}$ and $\psi_{g(a *_2 b)}=\psi_{g(b)} \psi_a \psi_b^{-1}$ for all $a, b \in X_2$, then $(X_1 \times X_2, \underline{*}, \overline{*})$ is a biquandle.
\end{enumerate}
\end{theorem}
\begin{proof}We will prove in details only the first assertion. The second assertion is similar. From the multiplication rules (\ref{productrule1}), (\ref{productrule2}) and the axiom (q1) in quandles $Q_1$, $Q_2$ for all $(x, a) \in X_1 \times X_2$ we have the equality
$$(x, a) \underline{*} (x, a)= \big(f(x), \phi_x(a) \big) =(x, a) \overline{*} (x, a),$$
and the first axiom of biquandles holds. 

For $(y,b)\in X_1\times X_2$ denote by $\alpha_{(y,b)}$ the map $\alpha_{(y,b)}:X_1\times X_2\to X_1\times X_2$ given by 
$$\alpha_{(y,b)}(x,a)=(x,a)\underline{*}(y,b)=(f(x*_1y),\phi_y(a)).$$
 Let us prove that $\alpha_{(y,b)}$ is bijective. If $\alpha_{(y, b)} (x_1, a_1) =\alpha_{(y, b)} (x_2, a_2)$
for some $(x_1,a_1), (x_2,a_2)$ from $X_1\times X_2$, then $f(x_1 *_1 y)=f(x_2 *_1 y)$ and $\phi_y(a_1)=\phi_y(a_2)$. By the injectivity of the maps $f  S_y$ and $\phi_y$ it is possible if and only if $(x_1, a_1)=(x_2, a_2)$, and therefore $\alpha_{(y,b)}$ is injective. For surjectivity, let $(z,c)$ be some element from $X_1 \times X_2$. Since $f  S_y$ and $\phi_y$ are surjective, there exist $(x, a) \in X_1 \times X_2$ such that $\alpha_{(y, b)} (x, a) = (z, c)$, hence $\alpha_{(y,b)}$ is surjective, and therefore is bijective. The fact that the maps $\beta_{(y, b)}$ defined by $\beta_{(y,b)}(x,a)=(x,a)\overline{*}(y,b)$ for $(x,a)\in X_1\times X_2$ is bijective can be proved similarly.

Consider the map $S: (X_1 \times X_2) \times (X_1 \times X_2) \to (X_1 \times X_2) \times (X_1 \times X_2)$ given by 
\begin{equation}\label{Sbijective}
S\big((x,a),(y,b)\big)=\big((y,b)\overline{*}(x,a),(x,a)\underline{*}(y,b)\big)=\big((f(y),\phi_x(b*_2a)),(f(x*_1y),\phi_y(a))\big),
\end{equation}
and let us prove that $S$ is bijective.  If
$S\big((x_1, a_1), (y_1, b_1) \big)=S\big(  (x_2, a_2), (y_2, b_2) \big)$ for some elements $x_1,x_2,y_1,y_2\in X_1$, $a_1,a_2,b_1,b_2\in X_2$, then from (\ref{Sbijective}) we have
\begin{align}
\notag f(y_1)&=f(y_2),&\phi_{x_1}(b_1 *_2 a_1)&=\phi_{x_2}(b_2 *_2 a_2),\\
\notag f(x_1 *_1 y_1)&= f(x_2 *_1 y_2),&\phi_{y_1}(a_1)&=\phi_{y_2}(a_2).
\end{align}
From the injectivity of maps $f$, $fS_{y_1}$, $fS_{y_2}$, $\phi_{y_1}$, $\phi_{y_2}$, $\phi_{x_1}S_{a_1}$, $\phi_{x_2}S_{a_2}$ follows that $(x_1, a_1)=(x_2, a_2)$ and $(y_1, b_1) =(y_2, b_2)$. Surjectivity of $S$ follows directly from equality (\ref{Sbijective}). So, the map $S$ is bijective, and the second axiom of biquandles holds. 

Finally, let us check the third axiom of biquandles. For $x,y,z\in X_1$ and $a,b,c\in X_2$, we have
\begin{align}
\notag \big( (x, a) \underline{*} (y, b) \big) \underline{*} \big( (z, c) \underline{*} (y, b) \big) &= \big( f(x *_1 y), \phi_y (a) \big) \underline{*} \big(f(z *_1 y), \phi_y (c) \big)\\
\notag &= \Big( f \big(f(x *_1 y) *_1 f(z *_1 y)\big), \phi_{f(z *_1 y)} \big(\phi_y(a)\big) \Big)\\
\notag&= \Big( f^2 \big((x *_1 z)*_1 y)\big), \phi_{f(z *_1 y)} \big(\phi_y(a)\big) \Big),~\text{by axiom (r2) in}~Q_1\\
\notag&= \Big( f^2 \big((x *_1 z)*_1 y)\big), \phi_{f(y)} \big(\phi_z(a)\big) \Big),~\textrm{since}~\phi_{f(z *_1y)}=\phi_{f(y)} \phi_z \phi_y^{-1}\\
\notag&=  \big(f(x*_1 z), \phi_z(a)\big ) \underline{*} \big( f(y), \phi_z(b *_2 c)\big ) \\
\notag&= \big( (x, a) \underline{*} (z, c) \big) \underline{*} \big( (y, b) \overline{*} (z, c) \big),\\
\notag\big( (x, a) \underline{*} (y, b) \big) \overline{*} \big( (z, c) \underline{*} (y, b) \big) &= \big( f(x *_1 y), \phi_y (a) \big) \overline{*} \big(f(z *_1 y), \phi_y (c) \big)\\
\notag&= \Big( f^2 (x *_1 y), \phi_{f(z *_1 y)} \big(\phi_y(a *_2 c)\big) \Big)\\
\notag&=  \Big( f^2 (x *_1 y),  \phi_{f(y)}\big(\phi_z(a *_2 c)\big) \Big),~\textrm{since}~\phi_{f(z *_1y)}=\phi_{f(y)} \phi_z \phi_y^{-1}\\
\notag&=  \big(f(x), \phi_z(a *_2 c)\big ) \underline{*} \big( f(y), \phi_z(b *_2 c)\big ) \\
\notag&= \big( (x, a) \overline{*} (z, c) \big) \underline{*} \big( (y, b) \overline{*} (z, c) \big),\\
\notag\big( (x, a) \overline{*} (y, b) \big) \overline{*} \big( (z, c) \overline{*} (y, b) \big) &= \big( f(x), \phi_y (a *_2 b) \big) \overline{*} \big(f(z), \phi_y (c *_2 b) \big)\\
\notag&= \Big( f^2(x), \phi_{f(z)} \big(\phi_y((a *_2 b) *_2 (c *_2 b))\big) \Big)\\
\notag&= \Big( f^2(x), \phi_{f(z)} \big(\phi_y((a *_2 c) *_2 b)\big) \Big),~\textrm{by axiom (r2) in}~Q_2\\
\notag&= \Big( f^2(x), \phi_{f(y *_1 z)} \big(\phi_z((a *_2 c) *_2 b)\big) \Big),~\textrm{since}~\phi_{f(y *_1z)}=\phi_{f(z)} \phi_y \phi_z^{-1}\\
\notag&=  \big(f(x), \phi_z(a *_2 c)\big ) \overline{*} \big( f(y *_1 z), \phi_z(b)\big )\\
\notag&= \big( (x, a) \overline{*} (z, c) \big) \overline{*} \big( (y, b) \underline{*} (z, c) \big),
\end{align}
i.~e. the third axiom of biquandles holds and the theorem is proved.
\end{proof}
The biquandle constructed in Theorem~\ref{action-product-biquandle} is denoted by $B(Q_1~{}_{\phi}\times{}_{\psi}~Q_2)$. If both $\phi$ and $\psi$ are trivial, i.~e. $\phi(Q_1)=\{id\}$, $\psi(Q_2)=\{id\}$, then the biquandle $B(Q_1~{}_{\phi}\times{}_{\psi}~Q_2)$ is the product biquandle introduced in \cite{Horvat, Kamada}. Note that the associated quandle $\mathcal{Q}(B(Q_1~{}_{\phi}\times{}_{\psi}~Q_2))$ of the biquandle $B(Q_1~{}_{\phi}\times{}_{\psi}~Q_2)$ has the operation 
 $$(x, a) * (y, b)= \big(x *_1y, a *_2^{-1} b\big),$$ 
for $(x,a), (y,b)\in X_1\times X_2$, which does not depend on $\phi,\psi$. Thus, if $Q_1=T_n$, $Q_2=T_m$ are trivial quandles, then $\mathcal{Q}(B(Q_1~{}_{\phi}\times{}_{\psi}~Q_2))=T_{mn}$ is the trivial quandle. So, using Theorem~\ref{action-product-biquandle} one can find more biquandle structures on trivial quandles for Problem~\ref{classificationTrivial}.

 If $\phi$ is the trivial homomorphism, i.~e. $\phi(Q_1)=\{id\}$, then condition (2) of Theorem~\ref{action-product-biquandle}  holds since $\psi:Q_2\to {\rm Conj}_{-1}({\rm Aut}(Q_1))$ is a homomorphism, and we obtain the following corollary. 
 
\begin{corollary}\label{one-constant-action}
Let $Q_1=(X_1, *_1)$, $Q_2=(X_2, *_2)$ be quandles, and $\psi: Q_2 \to {\rm Conj}_{-1} ({\rm Aut} (Q_1))$ be a quandle homomorphism. Then the set $X_1\times X_2$ with the operations
\begin{align}
\notag(x, a) \underline{*} (y, b)= \big(\psi_b(x*_1y), a \big)\\
\notag(x, a) \overline{*} (y, b)= \big(\psi_b(x), a *_2 b\big)
\end{align}
for $(x,a), (y,b)\in X_1\times X_2$ is a biquandle.
\end{corollary}
The biquandle obtained in Corollary~\ref{one-constant-action} is denoted by $B(Q_1\times_{\psi}Q_2)$. This biquandle is a kind of semi-direct product of biquandles $\mathcal{B}(Q_1)$ and $\mathcal{B}(Q_2)$. 
\begin{problem}Give a general definition of a semidirect product of (bi)quandles.
\end{problem}

\begin{remark}\label{rem-holomorph}
Let $Q=(X, *)$ be a quandle, $P={\rm Conj}_{-1}({\rm Aut}(Q))$, and $\psi:P\to{\rm Conj}_{-1}({\rm Aut}(Q))$ be the identity homomorphism. Using Corollary~\ref{one-constant-action} we can define the biquandle $$B(Q\times_{\psi}P)=B(Q\times_{\psi}{\rm Conj}_{-1}({\rm Aut}(Q))).$$ This biquandle is defined on the set $Q\times {\rm Aut}(Q)$ and has the following operations
\begin{align}
\notag(x, f) \underline{*} (y, g)=\big(g(x*y), f\big),&&(x, f) \overline{*} (y, g)=\big(g(x), gfg^{-1} \big)
\end{align} 
for $x, y \in Q$, $f, g \in {\rm Aut}(Q)$. We call this biquandle the \textit{holomorph biquandle} of $Q$ and denote it by ${\rm Hol}(Q)$.
\end{remark}

\subsection{Lifting biquandle structures}\label{liftingconstruction} 

Let $p:\widetilde{Q}\to Q$ be a quandle homomorphism. In this section we investigate when biquandle structures of $Q$ under certain conditions can be lifted to biquandle structures of $\widetilde{Q}$.

Inspired by the classical covering space theory for topological spaces,  Eisermann \cite{Eisermann} introduced the notion of a quandle covering and used it to study knot quandles.
A quandle homomorphism $p:\widetilde{Q} \to  Q$ is called a \textit{quandle covering} if it is surjective and the equality $p(\tilde{x}) = p(\tilde{y})$ for $\tilde{x},\tilde{y}\in \widetilde{Q}$ implies the equality $S_{\tilde{x}}= S_{\tilde{y}}$. Equivalently, a surjective quandle homomorphism  $p:\widetilde{Q} \to  Q$  is a quandle covering if and only if the natural map $S: \widetilde{Q} \to  \Inn(\widetilde{Q})$ factors through $p$.
\begin{example}For any quandle $Q$, the natural map $S: Q \to {\rm Conj}_{-1}({\rm Inn}(Q))$ which maps $x\in Q$ to $S_x\in {\rm Inn}(Q)$ defines a quandle covering $Q\to S(Q)$. By the definition, $S(Q)$ is the smallest quandle covered by $Q$.
\end{example}
\begin{example} If $G,\widetilde{G}$ are groups, then a surjective homomorphism $p : \widetilde{G} \to  G$ gives a quandle covering ${\rm Conj}(\widetilde{G}) \to {\rm Conj}(G)$ if and only if ${\rm Ker}(p)$ is a central subgroup of $\widetilde{G}$.
\end{example}
\begin{example} If $G,\widetilde{G}$ are groups, then a surjective group homomorphism $p : \widetilde{G} \to  G$ gives a quandle covering ${\rm Core}(\widetilde{G}) \to {\rm Core}(G)$ if and only if ${\rm Ker}(p)$ is a central subgroup of $\widetilde{G}$ of exponent~$2$.
\end{example}
\begin{example}
Let $\widetilde{G}$, $G$ be groups with $\widetilde{\phi} \in {\rm Aut}(\widetilde{G})$, $\phi \in {\rm Aut}(G)$. Then a surjective group homomorphism $p : \widetilde{G} \to  G$ with $p \tilde{\phi} = \phi  p$ gives a quandle covering ${\rm Alex}(\widetilde{G} ,\widetilde{\phi}) \to {\rm Alex}(G,\phi)$ if and only if $\widetilde{\phi}$ acts trivially on ${\rm Ker}(p)$.
\end{example}

Analogous to the topological covering space theory, Eisermann \cite{Eisermann} proved a lifting criteria for quandle coverings by introducing the fundamental group of a quandle. Recall that for each quandle $Q$, there is a natural quandle homomorphism 
$$\eta: Q \to {\rm Conj}_{-1} ({\rm Adj}(Q))$$
 which sends an element of $Q$ to a corresponding generator of ${\rm Adj}(Q)$. This gives a unique group homomorphism 
$$\varepsilon: {\rm Adj}(Q) \to \mathbb{Z}$$ 
with $\varepsilon(\eta(Q)) = 1$. Denote by ${\rm Adj}(Q)^\circ$ the kernel of the map $\varepsilon$. If $Q$ is connected, then $\varepsilon$ is the abelianization map, and hence ${\rm Adj}(Q)^\circ$ is the commutator subgroup of ${\rm Adj}(Q)$.

Let $Q$ be a quandle, and $q$ be an element of $Q$. The pair $(Q,q)$ is called the \textit{pointed quandle}. By the definition, a homomorphism $f:(Q_1,q_1)\to (Q_2,q_2)$ between pointed quandles $(Q_1,q_1)$, $(Q_2,q_2)$ is a homomorphism $f:Q_1\to Q_2$ such that $f(q_1)=q_2$. For a quandle $Q$ and an element $q\in Q$ denote by $\pi_1(Q,q)$ the following group
$$\pi_1(Q,q)=\{g \in {\rm Adj}(Q)^\circ~|~g\cdot q =q\}.$$
The group $\pi_1(Q,q)$ is called the \textit{fundamental group} of $Q$ at the based point $q\in Q$. In \cite{Eisermann} it is shown that the map $(Q, q) \mapsto \pi_1(Q, q)$ is a covariant functor from the category of pointed quandles to the category of groups. In particular, a homomorphism $f: (Q_1 , q_1) \to (Q_2 , q_2)$ of pointed quandles induces a group homomorphism  $f_*: \pi_1(Q_1 , q_1) \to \pi_1(Q_2 , q_2)$.  Moreover, if $Q$ is connected, then the isomorphism class of the group $\pi_1(Q,q)$ is independent of the choice of a base point $q \in Q$. The following result is proved in \cite[Proposition 4.9 and Proposition 5.13]{Eisermann}

\begin{theorem}\label{lifting theorem}
 Let $p: \widetilde{Q}\to Q$ be a quandle covering,  $\tilde{q}\in \widetilde{Q}$, and $q=p(\tilde{q})$. Let $X$ be a connected quandle, $x\in X$, and $f: (X,x) \to (Q,q)$ be a quandle homomorphism. Then there exists a unique lifting $\tilde{f}: (X,x) \to (\widetilde{Q},\tilde{q})$  such that $f=p\tilde{f}$,  if and only if $f_* \big(\pi_1( X , x) \big) \subset p_* \big(\pi_1 (\widetilde{Q}, \tilde{q}) \big)$.
\end{theorem}

A quandle $Q$ is called \textit{simply connected} if it is connected and $\pi_1(Q,q) = \{1\}$. For example,  every dihedral quandle $\R_n$ of odd order is simply connected \cite[Example 1.17]{Eisermann}. Also, the knot quandle $Q_L$ of a long knot $L$ is simply connected \cite[Theorem 30]{Eisermann-2003}.
\begin{theorem}\label{lifting-biquandle-structure}
Let $p:\widetilde{Q} \to  Q$ be a quandle covering. If $\widetilde{Q}$ is simply connected, then every biquandle structure  $\{\beta_y~|~ y \in Q\}$  on $Q$ lifts to a biquandle structure $\{\alpha_{\tilde{y}}~|~ \tilde{y} \in \widetilde{Q}\}$ on $\widetilde{Q}$.
\end{theorem}
\begin{proof}
Let $y$ be a fixed element of $Q$. Since $p$ is surjective, there exist $\tilde{y}, \tilde{z}  \in \widetilde{Q}$ such that $p(\tilde{y})=y$ and $p(\tilde{z})=\beta_y(y)$. Consider the maps  
\begin{align}
\notag \beta_yp&: (\widetilde{Q}, \tilde{y}) \to (Q, \beta_y(y))\\
\notag p&:(\widetilde{Q}, \tilde{z}) \to (Q, \beta_y(y))
\end{align} 
of pointed quandles. Since  $(\beta_y p)_* \big(\pi_1( \widetilde{Q} , \tilde{y}) \big) =\{1\}= p_* \big(\pi_1 (\widetilde{Q}, \tilde{z}) \big)$, by Theorem \ref{lifting theorem}, there exists a unique quandle homomorphism $\widetilde{\beta_y}: (\widetilde{Q}, \tilde{y}) \to  (\widetilde{Q}, \tilde{z})$ such that $\beta_yp=  p\widetilde{\beta_y}$. Similarly, for 
$$\beta_y^{-1} p: (\widetilde{Q},\tilde{z}) \to (Q, y)$$ 
there exists a unique homomorphism $\widetilde{\beta_y^{-1}}: (\widetilde{Q}, \tilde{z}) \to (\widetilde{Q}, \tilde{y})$ with $\beta_y^{-1} p= p \widetilde{\beta_y^{-1}}$. Thus, we have 
$$p (\widetilde{\beta_y^{-1}} \widetilde{\beta_y})=(\beta_y^{-1} p) \widetilde{\beta_y}= \beta_y^{-1} (\beta_yp)= p,$$
i.~e. $(\widetilde{\beta_y^{-1}} \widetilde{\beta_y})$ is a lift of $id$ (on $Q$) with respect to $p$. Since $id$ (on ${\widetilde{Q}}$) is also a lift of $id$ (on $Q$) with respect to $p$, it follows from the uniqueness of the lift that $\widetilde{\beta_y^{-1}} \widetilde{\beta_y} = id =\widetilde{\beta_y}  \widetilde{\beta_y^{-1}}$, and hence $\widetilde{\beta_y} \in {\rm Aut}(\widetilde{Q})$.

For each $\tilde{y} \in p^{-1}(y)$ denote by $\alpha_{\tilde{y}}=\widetilde{\beta_y}$ and let us prove that $\{\alpha_{\tilde{y}}~|~\tilde{y}\in\widetilde{Q}\}$ is a biquandle structure on $\widetilde{Q}$. By construction, we have 
\begin{equation}\label{lifting-condition}
\beta_y p=  p \alpha_{\tilde{y}}
\end{equation}
 for all $y \in Q$ and $\tilde{y} \in \widetilde{Q}$ with $p(\tilde{y})=y$. Let $\tilde{x} \neq \tilde{y}$ be elements of $\widetilde{Q}$. If $p(\tilde{x})=p(\tilde{y})$, then by definition, $\alpha_{\tilde{x}}=\alpha_{\tilde{y}}$, and hence $\alpha_{\tilde{x}}(\tilde{x}) \neq \alpha_{\tilde{y}}(\tilde{y})$. Suppose that $p(\tilde{x}) \neq p(\tilde{y})$. Set $x=p(\tilde{x})$ and $y=p(\tilde{y})$. Since the map $z \mapsto \beta_z(z)$ is injective, we have $\beta_x(x) \neq \beta_y(y)$. But, 
$$p \widetilde{\beta_x}(\tilde{x})= \beta_x p (\tilde{x})=\beta_x(x)\neq \beta_y(y)= \beta_y p (\tilde{y})=p \widetilde{\beta_y}(\tilde{y}).$$
This implies that $\widetilde{\beta_x}(\tilde{x}) \neq \widetilde{\beta_y}(\tilde{y})$, that is, $\alpha_{\tilde{x}}(\tilde{x}) \neq \alpha_{\tilde{y}}(\tilde{y})$. Hence, the map $\tilde{x} \mapsto \alpha_{\tilde{x}}(\tilde{x})$ is injective.

Let $\tilde{y} \in \widetilde{Q}$ and set $y=p(\tilde{y})\in Q$. Since the map $z \mapsto \beta_z(z)$ is surjective, there exists $x \in Q$ such that $\beta_x(x)=y$. Note that a lift $\widetilde{\beta_x}$ of $\beta_x$ induces a bijection between the fibers $p^{-1}(x)$ and $p^{-1}(y)$. Thus, for 
$\tilde{y} \in p^{-1}(y)$, there exists an element $\tilde{x} \in p^{-1}(x)$ such that $\alpha_{\tilde{x}}(\tilde{x})= \tilde{y}$. Hence, the map $\tilde{y}\mapsto\alpha_{\tilde{y}}(\tilde{y})$ is surjective, and therefore is bijective. In order to prove that $\{\alpha_{\tilde{y}}~|~\tilde{y}\in\widetilde{Q}\}$ is a biquandle structure on $\widetilde{Q}$ we need to prove that 
$$\alpha_{\alpha_{\tilde{y}}(\tilde{x}*\tilde{y})} \alpha_{\tilde{y}} = \alpha_{ \alpha_{\tilde{x}}(\tilde{y})}  \alpha_{\tilde{x}}$$
for all $\tilde{x},\tilde{y}\in\widetilde{Q}$. Let $\tilde{x}, \tilde{y} \in \widetilde{Q}$ and set $x=p(\tilde{x})$, $y=p(\tilde{y})$. Then, we have
\begin{align}
\notag p \big(\alpha_{\alpha_{\tilde{y}}(\tilde{x}*\tilde{y})}  \alpha_{\tilde{y}} \big) &= \big(p \alpha_{\alpha_{\tilde{y}}(\tilde{x}*\tilde{y})} \big)\alpha_{\tilde{y}}= \big(\beta_{ p \alpha_{\tilde{y}}(\tilde{x}*\tilde{y})} p \big) \alpha_{\tilde{y}} = \beta_{ \beta_y p(\tilde{x}*\tilde{y})} \big( p \alpha_{\tilde{y}} \big)= \beta_{ \beta_y (x*y)} \big( \beta_y p \big)\\
\notag&= \beta_{ \beta_x (y)} \big( \beta_x p \big),~\textrm{since}~\{\beta_y~|~y\in Q\}~\textrm{is a biquandle structure on}~Q\\
\notag&= \beta_{ \beta_x (y)} \big(p \alpha_{\tilde{x}} \big)= \big( \beta_{ \beta_x (y)} p  \big)\alpha_{\tilde{x}}= \big( \beta_{ \beta_x p (\tilde{y})} p  \big) \alpha_{\tilde{x}}= \big( \beta_{ p \alpha_{\tilde{x}} (\tilde{y})} p  \big) \alpha_{\tilde{x}}= \big( p \alpha_{\alpha_{\tilde{x}}(\tilde{y})}  \big) \alpha_{\tilde{x}}\\
\notag &= p \big(\alpha_{ \alpha_{\tilde{x}}(\tilde{y})} \alpha_{\tilde{x}} \big).
\end{align} 
By uniqueness of the lift with respect to the covering $p$, we have $\alpha_{\alpha_{\tilde{y}}(\tilde{x}*\tilde{y})} \alpha_{\tilde{y}} =\alpha_{ \alpha_{\tilde{x}}(\tilde{y})} \alpha_{\tilde{x}}$, which completes the proof.
\end{proof}
The following result describes connections between biquandles obtained from $Q,\widetilde{Q}$ using biquandle structures obtained in Theorem~\ref{lifting-biquandle-structure}.
\begin{proposition}
Let $\widetilde{Q}$ be a simply connected quandle, and $p:\widetilde{Q} \to  Q$ be a quandle covering. Let $A=\{\beta_y~|~ y \in Q\}$ be a biquandle structure on $Q$, and $\widetilde{A}=\{\alpha_{\tilde{y}}~|~\tilde{y}\in\widetilde{Q}\}$ be a biquandle structure on $\widetilde{Q}$ constructed in Theorem~\ref{lifting-biquandle-structure}. Let $B,\widetilde{B}$ be biquandles constructed from quandles $Q,\widetilde{Q}$ using biquandle structures $A,\widetilde{A}$, respectively. Then $p$ induces a biquandle homomorphism $p:\widetilde{B}\to B$.
\end{proposition}
\begin{proof}The operations on biquandles $\widetilde{B}, B$ are defined by the rules
\begin{align}
\notag\tilde{x} \underline{*} \tilde{y}= \alpha_{\tilde{y}}(\tilde{x} *\tilde{y}),&&\tilde{x} \overline{*} \tilde{y}= \alpha_{\tilde{y}}(\tilde{x}),&& x \underline{*} y= \beta_y(x *y),&&x \overline{*} y= \beta_y(x)
\end{align} 
for $\tilde{x}, \tilde{y} \in \widetilde{Q}$ and $x, y \in Q$. The equalities
\begin{align}
\notag p(\tilde{x} \underline{*} \tilde{y}) &= p\big(\alpha_{\tilde{y}}(\tilde{x} *\tilde{y})\big)= \beta_y p (\tilde{x} * \tilde{y})= \beta_y  (x* y)= x\underline{*} y = p(\tilde{x}) \underline{*} p(\tilde{y}),\\
\notag p(\tilde{x} \overline{*} \tilde{y}) &= p\big( \alpha_{\tilde{y}}(\tilde{x})\big)= \beta_y p (\tilde{x})= \beta_y  (x)= x\overline{*} y= p(\tilde{x}) \overline{*} p(\tilde{y})
\end{align}
prove that $p:\widetilde{B}\to B$ is a biquandle homomorphism.
\end{proof}
Let $X=\{x_1,\dots,x_n\}$, and $FQ_n$ be the free quandle on $X$. The quandle $FQ_n$ can be written as the disjoint union of orbits $$FQ_n={\rm Orb}(x_1)\sqcup {\rm Orb}(x_2)\sqcup \dots\sqcup {\rm Orb}(x_n)$$ There exists a natural homomorphism $p:FQ_n\to T_n$ from the free quandle $FQ_n$ to the trivial quandle $T_n=\{t_1,\dots,t_n\}$ defined by the rule $p(a)=t_i$ if and only if $a\in {\rm Orb}(x_i)$. Unfortunately, $T_n$ is not a simply connected quandle, and $p:FQ_n\to T_n$ is not a quandle covering, so, Theorem~\ref{lifting-biquandle-structure} does not give a way to lift biquandle structures from $T_n$ to $FQ_n$. However, due to importance of the free quandle $FQ_n$, it would be useful to find a procedure to lift biquandle structures from $T_n$ to $FQ_n$. Let us describe one such procedure.

Let $\{\beta_1=\beta_{t_1}, \beta_2=\beta_{t_2},\dots,\beta_n=\beta_{t_n}\}$ be a biquandle structure on $T_n$. Since ${\rm Aut}(T_n)=\Sigma_n$, the automorphisms $\beta_1,\beta_2,\dots,\beta_n$ are just permutations of $\{t_1,t_2,\dots,t_n\}$. Thus, we can think about these permutations as about permutations of $\{1,2,\dots,n\}$ and write $\beta_i(t_j)=t_{\beta_i(j)}$ for $i,j=1,2,\dots,n$. For $a\in FQ_n$ denote by $\alpha_a$ the automorphism of $FQ_n$ induced by the permutation of generators $\{x_1,x_2,\dots,x_n\}$ of the form: if $a\in {\rm Orb}(x_i)$, then $\alpha_a(x_j)=x_{\beta_i(j)}$ (so, in this situation we can write $\alpha_a=\beta_i$). The following proposition states that $\{\alpha_a~|~a\in FQ_n\}$ is a biquandle structure of $FQ_n$.
\begin{proposition}\label{lifting}
If $\{\beta_1,\dots,\beta_n\}$ is a biquandle structure on $T_n$, then $\{\alpha_a~|~a\in FQ_n\}$ is a biquandle structure of $FQ_n$.
\end{proposition}

\begin{proof} We need to check the following two conditions
\begin{enumerate}
\item $\alpha_{\alpha_y(x*y)}\alpha_y=\alpha_{\alpha_x(y)}\alpha_x$ for all $x,y\in FQ_n$,
\item the map $y\mapsto \alpha_y(y)$ is a bijection of $FQ_n$. 
\end{enumerate}
Let $x\in {\rm Orb}(x_i)$, $y\in {\rm Orb}(x_j)$, $\alpha_y(x)\in {\rm Orb}(x_r)$, and $\alpha_x(y)\in {\rm Orb}(x_s)$. From these four conditions and the definition of $\alpha_a$ (for $a\in FQ_n$) follows that $\beta_j(t_i)=t_r$, $\beta_i(t_j)=t_s$. Since $\{\beta_1=\beta_{t_1}, \beta_2=\beta_{t_2},\dots,\beta_n=\beta_{t_n}\}$ is a biquandle structure on $T_n$, we have the equality 
$$\beta_{\beta_{t_j}(t_i*t_j)}\beta_{t_j}=\beta_{\beta_{t_i}(t_j)}\beta_{t_i}$$
Since $t_i*t_j=t_i$, we can rewrite the last equality in the following form 
\begin{equation}\label{ontriv}
\beta_r\beta_j=\beta_s\beta_{i}
\end{equation}
Since $\alpha_y(x)\in {\rm Orb}(x_r)$, we have $\alpha_y(x*y)\in {\rm Orb}(x_r)$ and 
\begin{align}
\label{leftbqs}\alpha_{\alpha_y(x*y)}\alpha_y&=\alpha_{x_r}\alpha_{x_j}=\beta_r\beta_j,\\
\label{rightbqs}\alpha_{\alpha_x(y)}\alpha_x&=\alpha_{x_s}\alpha_{x_i}=\beta_{s}\beta_i.
\end{align}
From equalities (\ref{ontriv}), (\ref{leftbqs}), (\ref{rightbqs}) follows that the equality 
$$\alpha_{\alpha_y(x*y)}\alpha_y=\alpha_{\alpha_x(y)}\alpha_x$$
holds for all $x,y\in FQ_n$, and the first axiom of a biquandle structure holds for $\{\alpha_a~|~a\in FQ_n\}$. Let us check the second axiom of a biquandle structre, i.~e. that the map $x\mapsto \alpha_x(x)$ is bijective. Let us prove that this map is injective. Suppose that there exist $x,y\in FQ_n$ such that 
$$\alpha_x(x)=\alpha_y(y).$$
If ${\rm Orb}(x)={\rm Orb}(y)$, then $\alpha_x=\alpha_y$, and from equality $\alpha_x(x)=\alpha_y(y)$ follows that $x=y$. If ${\rm Orb}(x)\neq{\rm Orb}(y)$, then from the fact $\{\beta_1,\dots,\beta_n\}$ is a biquandle structure on $T_n$ and the definition of $\alpha_a$ (for $a\in FQ_n$) follows that ${\rm Orb}(\alpha_x(x))\neq {\rm Orb}(\alpha_y(y))$ and $\alpha_x(x)\neq\alpha_y(y)$. Thus, the map $x\mapsto \alpha_x(x)$ is injective. 

Let us prove that the map $x\mapsto \alpha_x(x)$ is surjective. Let $x\in {\rm Orb}(x_i)$, then 
$$x=x_i*^{\varepsilon_1}y_1*^{\varepsilon_2}y_2*^{\varepsilon_3}\dots*^{\varepsilon_m}y_m$$
for some $x_i,y_1,y_2,\dots,y_m\in \{x_1,\dots,x_n\}$, $\varepsilon_1,\varepsilon_2,\dots,\varepsilon_n\in\{\pm1\}$. Since $\{\beta_1,\dots,\beta_n\}$ is a biquandle structure on $T_n$ there exists $t_j$ such that $\beta_j(t_j)=t_i$. For $i=1,\dots,m$, denote by $z_i=\beta_j^{-1}(y_i)$ (it is clear that $z_i\in\{x_1,\dots,x_n\}$), and let 
$$b=x_j*^{\varepsilon_1}z_1*^{\varepsilon_2}z_2*^{\varepsilon_3}\dots*^{\varepsilon_m}z_m.$$
Then from the definition of $\alpha_a$ (for $a\in FQ_n$) follows that
\begin{align}
\notag\alpha_b(b)&=\alpha_b(x_j)*^{\varepsilon_1}\alpha_b(z_1)*^{\varepsilon_2}\alpha_b(z_2)*^{\varepsilon_3}\dots*^{\varepsilon_m}\alpha_b(z_m)\\
\notag&=\alpha_{x_j}(x_j)*^{\varepsilon_1}\alpha_{x_j}(z_1)*^{\varepsilon_2}\alpha_{x_j}(z_2)*^{\varepsilon_3}\dots*^{\varepsilon_m}\alpha_{x_j}(z_m)\\
\notag&=x_i*^{\varepsilon_1}y_1*^{\varepsilon_2}y_2*^{\varepsilon_3}\dots*^{\varepsilon_m}y_m=x,
\end{align}
i.~e. the map $x\mapsto\alpha_x(x)$ is surjective, and therefore is  bijective. Thus, we proved that $\{\beta_a~|~a\in FQ_n\}$ is a biquandle structure of $FQ_n$.
\end{proof}

From Proposition~\ref{lifting} follows that every biquandle structure on a trivial quandle $T_n$ can be lifted to a biquandle structure on the free quandle $FQ_n$, which gives additional motivation for studying Problem~\ref{classificationTrivial}. Furthermore, in \cite{Fenn} the question of giving an explicit model for a free biquandle is posed. We believe that a free biquandle should be obtainable using a biquandle structure on a free quandle.

\section{Automorphisms of biquandles}\label{sec-automorphisms}

Let $B=(X,\underline{*},\overline{*})$ be a biquandle, and $Q=\mathcal{Q}(B)=(X,*)$ be the associated quandle of $B$. Since the equality $x*y=(x\underline{*}y)\overline{*}^{-1}y$ holds for all $x,y\in X$, every automorphism of $B$ induces an automorphism of $Q$, and we have the inclusion ${\rm Aut}(B)\leq {\rm Aut}(Q)$. The following result proved in \cite[Theorem~4.1]{Horvat} gives more information about this inclusion. 
\begin{theorem}\label{biquandlestructureautomorphism}Let $B$ be a biquandle obtained from a quandle $Q$ using a biquandle structure $\{\beta_a~|~a\in Q\}\subset {\rm Aut}(Q)$. Then ${\rm Aut}(Q)\leq N_{{\rm Aut}(Q)}\{\beta_a~|~a\in Q\}$.
\end{theorem}
If $B$ is a biquandle obtained from a quandle $Q$ using a constant biquandle structure, then the group of automorphism of $B$ is completely described in the following proposition proved in \cite[Corollary~4.2]{Horvat}.
\begin{proposition}\label{constantbiqstr}Let $Q$ be a quandle, and $B$ be a biquandle obtained from $Q$ using a constant biquandle structure $\{\beta_a=f~|~a\in Q\}$ for $f\in {\rm Aut}(Q)$. Then ${\rm Aut}(B)=C_{{\rm Aut}(Q)}(f)$.
\end{proposition}

Let us find automorphism groups of biquandles introduced in Proposition~\ref{gen-dihedral-biquandle}. The following result is a partial generalisation of \cite[Proposition~4.5]{Horvat}.

\begin{proposition}\label{auto-gen-dihedral-biquandle}
Let $G$ be an abelian group without $2$-torsion and $\phi$ be an automorphism of $G$. Then $C_{{\rm Aut}({\rm T}(G))}(\phi) \leq {\rm Aut}\big(B(G, \phi) \big)$.
\end{proposition}

\begin{proof} A direct check shows that the associated quandle $\mathcal{Q}(B(G,\phi))$ is the Takasaki quandle ${\rm T}(G)$. By \cite[Theorem 4.2(1)]{BDS}, ${\rm Aut}({\rm T}(G)) =G \rtimes {\rm Aut}(G)$, and for every automorphism $f$ of ${\rm T}(G)$ there exist $g \in G$, $\alpha \in {\rm Aut}(G)$ such that
$f(x)=g\alpha(x)$
for all $x\in G$. Let $f\in C_{{\rm Aut}({\rm T}(G))}(\phi)$ and $x, y \in B(G, \phi)$. Then
\begin{align}
\notag f(x \underline{*} y) &= g \alpha \big( \phi(y) \big) \alpha (x)^{-1} \alpha(y)= f \big( \phi(y) \big) \alpha (x)^{-1} \alpha(y) \\
\notag &= \phi\big(f(y)\big)\alpha(x)^{-1}g^{-1}g\alpha(y)= \phi\big(f(y)\big)f(x)^{-1} f(y)=f(x) \underline{*} f(y)
\end{align}
and
$f(x \overline{*} y) =f\big(\phi(x)\big)=\phi\big(f(x)\big)=f(x) \overline{*} f(y)$. Thus, $f$ is an automorphism of $B(G,\phi)$, and we have a map $C_{{\rm Aut}({\rm T}(G))}(\phi) \to {\rm Aut}\big( B(G, \phi) \big)$ given by $f \mapsto f$, which is clearly an embedding.
\end{proof}

The following result is an improvement of \cite[Corollary 4.3]{Horvat} for finite groups.

\begin{proposition}\label{auto-gen-alexander}
Let $G$ be a finite abelian group, and $\phi, \psi$ be two commuting automorphisms of $G$ such that $\psi^{-1}\phi$ is a fixed point free automorphism. Then ${\rm Aut}\big( {\rm A}_{\psi,\phi}(G)\big) ={\rm Fix}(\psi) \rtimes C_{{\rm Aut}(G)}( \phi, \psi)$, where ${\rm Fix}(\psi)$ is the group of fixed-points of $\psi$.
\end{proposition}

\begin{proof}Direct check shows that $\mathcal{Q}({\rm A}_{\psi,\phi})$ is the Alexander quandle ${\rm Alex}(G,\psi^{-1}\phi)$, and the biquandle  ${\rm A}_{\psi,\phi}$ can be obtained from ${\rm Alex}(G,\psi^{-1}\phi)$ using a constant biquandle structure $\{\beta_x=\psi~|~x\in G\}$. Since $\psi^{-1}\phi$ is a fixed point free automorphism, from \cite[Theorem 6.1]{BDS} follows that ${\rm Aut} \big({\rm Alex}(G, \psi^{-1}\phi ) \big) = G \rtimes C_{{\rm Aut}(G)}(\psi^{-1}\phi)$. Since ${\rm A}_{\psi,\phi}$ is obtained from ${\rm Alex}(G,\psi^{-1}\phi)$ using a constant biquandle structure $\{\beta_x=\psi~|~x\in G\}$ from Proposition~\ref{constantbiqstr} and direct calculations follow that
$${\rm Aut} \big({\rm A}_{\psi,\phi} \big) =C_{{\rm Aut} \big({\rm Alex}(G, \psi^{-1}\phi ) \big) }(\psi) =  {\rm Fix}(\psi) \rtimes C_{{\rm Aut}(G)}( \phi, \psi).$$
The statement is proved.
\end{proof}

In the remainder of this section we study automorphism groups of biquandles introduced in Section~\ref{newconstructionsnew} and their connections with automorphism groups of associated quandles.
\subsection{Automorphisms of union biquandles}At first, let us find the automorphism group of the union quandle $Q_1\sqcup Q_2$ for connected quandles $Q_1$, $Q_2$.
\begin{lemma}\label{unionaut}Let $Q_1=(X_1,*_1)$, $Q_2=(X_2,*_2)$ be connected quandles. 
\begin{enumerate}
\item If $Q_1\not\simeq Q_2$, then ${\rm Aut}(Q_1\sqcup Q_2)={\rm Aut}(Q_1)\times {\rm Aut}(Q_2)$.
\item If $\alpha:Q_1\to Q_2$ is an isomorphism, then ${\rm Aut}(Q_1\sqcup Q_2)=({\rm Aut}(Q_1)\times {\rm Aut}(Q_2))\rtimes_{\theta} \mathbb{Z}_2$, where $\theta(f_1,f_2)=(\alpha^{-1}f_2\alpha,\alpha f_1\alpha^{-1})$.
\end{enumerate}
\end{lemma}
\begin{proof} If $f_1\in {\rm Aut}(Q_1)$, $f_2\in {\rm Aut}(Q_2)$, then it is clear that the map $f:Q_1\sqcup Q_2\to Q_1\sqcup Q_2$ given by $f(x)=f_i(x)$ for $x\in Q_i~(i=1,2)$ is an automorphism of $Q$. Denote this automorphism by $f=(f_1,f_2)$. The set of all such automorphisms of $Q_1\sqcup Q_2$ forms a subgroup in ${\rm Aut}(Q_1\sqcup Q_2)$ which is isomorphic to ${\rm Aut}(Q_1)\times {\rm Aut}(Q_2)$. Further, an automorphism $f\in {\rm Aut}(Q_1\sqcup Q_2)$ belongs to ${\rm Aut}(Q_1)\times {\rm Aut}(Q_2)$ if and only if $f(Q_1)=Q_1$, $f(Q_2)=Q_2$.

(1) Let $Q_1\not\simeq Q_2$. Let us prove that in this case every automorphism $f$ of $Q_1\sqcup Q_2$ satisfies $f(Q_1)=Q_1$, $f(Q_2)=Q_2$.
By contrary, suppose that there exists an element $x\in Q_1$ such that $f(x)\in Q_2$. Since $Q_1$ is connected, for each $y\in Q_1$ there exist elements $x_1,\dots,x_n$ and integers $\varepsilon_1,\dots,\varepsilon_n\in\{\pm1\}$ such that $y=x*^{\varepsilon_1}x_1*^{\varepsilon_2}x_2*^{\varepsilon_3}\dots*^{\varepsilon_n}x_n$. Thus,
$$f(y)=f(x)*^{\varepsilon_1}f(x_1)*^{\varepsilon_2}f(x_2)*^{\varepsilon_3}\dots*^{\varepsilon_n}f(x_n)$$
belongs to $Q_2$, and hence $f(Q_1)\subset Q_2$. Since $f\in {\rm Aut}(Q_1\sqcup Q_2)$, the induced map $f:Q_1\to Q_2$ is an injective homomorphism. Let us prove that $f:Q_1\to Q_2$ is surjective.

By contrary, suppose that there exists an element $b\in Q_2$ such that $f(a)\neq b$ for all $a\in Q_1$. Since $f\in {\rm Aut}(Q_1\sqcup Q_2)$, there exists $c\in Q_2$ such that $f(c)=b$. Since $Q_2$ is connected, $f(Q_2)\subset Q_2$. Thus, $f(Q_1\sqcup Q_2)\subset Q_2$ which contradicts  the fact that  $f\in {\rm Aut}(Q_1\sqcup Q_2)$. Thus, $f:Q_1\to Q_2$ is a surjective (and therefore bijective) homomorphism. Thus, $Q_1\simeq Q_2$ which contradicts the given hypothesis.

(2) Let $\alpha:Q_1\to Q_2$ be an isomorphism. Denote by $\iota:Q_1\sqcup Q_2\to Q_1\sqcup Q_2$ the map given by
\begin{align}\label{iota}
\iota(x)=\begin{cases}
\alpha(x),& x\in Q_1,\\
\alpha^{-1}(x),&x\in Q_2.
\end{cases}
\end{align}
It is clear that $\iota$ is an automorphism of $Q_1\sqcup Q_2$ of order $2$. Let us prove that the group ${\rm Aut}(Q_1\sqcup Q_2)$ is generated by ${\rm Aut}(Q_1)\times {\rm Aut}(Q_2)$ and automorphism $\iota$. Let $f\in {\rm Aut}(Q_1\sqcup Q_2)$, and let $x$ be an element of $Q_1$. If $f(x)\in Q_1$, then similar to case (1) we have $f(Q_1)=Q_1$, $f(Q_2)=Q_2$, and therefore $f\in {\rm Aut}(Q_1)\times {\rm Aut}(Q_2)$. If $f(x)\in Q_2$, then similar to case (1) we have $f(Q_1)=Q_2$, $f(Q_2)=Q_1$. Thus, for the automorphism $g=\iota f$, we have $g(x)=\iota f(x)=\alpha^{-1}(f(x))\in Q_1$. Hence, $g(Q_1)=Q_1$, $g(Q_2)=Q_2$, i.~e. $g\in {\rm Aut}(Q_1)\times {\rm Aut}(Q_2)$, and $f=\iota g\in\langle{\rm Aut}(Q_1)\times {\rm Aut}(Q_2),\iota\rangle$. Hence, we proved that the group ${\rm Aut}(Q_1\sqcup Q_2)$ is generated by ${\rm Aut}(Q_1)\times {\rm Aut}(Q_2)$ and automorphism $\iota$. Let $f_1\in {\rm Aut}(Q_1)$, $f_2\in {\rm Aut}(Q_2)$. The equality
\begin{align}
\notag\iota(f_1,f_2)\iota(x)&=\begin{cases}
\iota(f_1,f_2)(\alpha(x)),&x\in Q_1\\
\iota(f_1,f_2)(\alpha^{-1}(x)),&x\in Q_2
\end{cases}\\
\notag&=\begin{cases}
\iota(f_2(\alpha(x))),&x\in Q_1\\
\iota(f_1(\alpha^{-1}(x))),&x\in Q_2
\end{cases}\\
\notag&=\begin{cases}
\alpha^{-1}f_2\alpha(x),&x\in Q_1\\
\alpha f_1\alpha^{-1}(x),&x\in Q_2
\end{cases}
\end{align}
implies that ${\rm Aut}(Q_1\sqcup Q_2)=({\rm Aut}(Q_1)\times {\rm Aut}(Q_2))\rtimes_{\theta} \mathbb{Z}_2$, with the action described in the formulation of the lemma.
\end{proof}
Let $Q_1=(X_1,*_1)$, $Q_2=(X_2,*_2)$ be quandles, $f_1\in {\rm Aut}(Q_1)$, and $f_2\in {\rm Aut}(Q_2)$. Recall that the biquandle $B(Q_1~{}_{f_2}\bigsqcup{}_{f_1}~Q_2)$ is a biquandle on the set $X_1\sqcup X_2$ which has the following operations:
\begin{align}
\notag a,b\in Q_1 &\Rightarrow a\overline{*}b=a, a\underline{*}b=a*_1b,\\
\notag a,b\in Q_2 &\Rightarrow a\overline{*}b=a, a\underline{*}b=a*_2b,\\
\notag a\in Q_1, b\in Q_2 &\Rightarrow a\overline{*}b=f_1(a), a\underline{*}b=f_1(a),\\
\notag a\in Q_2, b\in Q_1 &\Rightarrow a\overline{*}b=f_2(a), a\underline{*}b=f_2(a).
\end{align} 
(see Section~\ref{sunion}). We are ready to prove the following.

\begin{theorem}\label{auto-union-biquandle}
Let $Q_1=(X_1,*_1)$, $Q_2=(X_2,*_2)$ be connected quandles, $f_1\in {\rm Aut}(Q_1)$, and $f_2\in {\rm Aut}(Q_2)$.
\begin{enumerate}
\item If $Q_1\not\simeq Q_2$, then ${\rm Aut}\left(B(Q_1~{}_{f_2}\bigsqcup{}_{f_1}~Q_2)\right)=C_{{\rm Aut}(Q_1)}(f_1)\times C_{{\rm Aut}(Q_2)}(f_2)$. 
\item If $\alpha:Q_1\to Q_2$ is an isomorphism, and $f_1,\alpha^{-1} f_2\alpha$ are not conjugate in ${\rm Aut}(Q_1)$, then ${\rm Aut}(B(Q_1~{}_{f_2}\bigsqcup{}_{f_1}~Q_2))=C_{{\rm Aut}(Q_1)}(f_1)\times C_{{\rm Aut}(Q_2)}(f_2)$
\item If $\alpha:Q_1\to Q_2$ is an isomorphism, and $\alpha^{-1} f_2\alpha=\psi^{-1}f_1\psi$ for $\psi\in{\rm Aut}(Q_1)$, then ${\rm Aut}(B(Q_1~{}_{f_2}\bigsqcup{}_{f_1}~Q_2))=(C_{{\rm Aut}(Q_1)}(f_1)\times C_{{\rm Aut}(Q_2)}(f_2))\rtimes_{\rho}\mathbb{Z}_2$, where the action is given by $\rho(\varphi_1,\varphi_2)=((\alpha\psi^{-1})^{-1}\varphi_2(\alpha\psi^{-1}),(\alpha\psi^{-1}) \varphi_1(\alpha\psi^{-1})^{-1})$.
\end{enumerate}
\end{theorem}
\begin{proof}
(1) Let $Q_1\not\simeq Q_2$. If $\varphi_1\in C_{{\rm Aut}(Q_1)}(f_1)$, $\varphi_2\in C_{{\rm Aut}(Q_2)}(f_2)$, then a direct check shows that the map $\varphi=(\varphi_1,\varphi_2)$ given by $\varphi(x)=\varphi_i(x)$ for $x\in X_i$ is an automorphism of $B(Q_1~{}_{f_2}\bigsqcup{}_{f_1}~Q_2)$. The subgroup of ${\rm Aut}(B(Q_1~{}_{f_2}\bigsqcup{}_{f_1}~Q_2))$ generated by such automorphisms is isomorphic to $C_{{\rm Aut}(Q_1)}(f_1)\times C_{{\rm Aut}(Q_2)}(f_2)$. 

Let us prove that every automorphism $\varphi$ of $B(Q_1~{}_{f_2}\bigsqcup{}_{f_1}~Q_2)$ can be written as $\varphi=(\varphi_1,\varphi_2)$ for $\varphi_1\in C_{{\rm Aut}(Q_1)}(f_1)$, $\varphi_2\in C_{{\rm Aut}(Q_2)}(f_2)$. Let $\varphi\in {\rm Aut}(B(Q_1~{}_{f_2}\bigsqcup{}_{f_1}~Q_2))$. Since $\varphi$ induces an automorphism of $\mathcal{Q}(B(Q_1~{}_{f_2}\bigsqcup{}_{f_1}~Q_2))=Q_1\sqcup Q_2$, from Lemma~\ref{unionaut}(1) follows that $\varphi$ belongs to ${\rm Aut}(Q_1)\times {\rm Aut}(Q_2)$, i.~e. $\varphi=(\varphi_1,\varphi_2)$ for $\varphi_1\in {\rm Aut}(Q_1)$, $\varphi_2\in {\rm Aut}(Q_2)$. Since $\varphi$ is an automorphism of $B(Q_1~{}_{f_2}\bigsqcup{}_{f_1}~Q_2)$,  we have
$\varphi(a\overline{*}b)=\varphi(a)\overline{*}\varphi(b)$ for all $a,b\in B(Q_1~{}_{f_2}\bigsqcup{}_{f_1}~Q_2)$.
If $a\in Q_1$, $b\in Q_2$, then $\varphi(a)=\varphi_1(a)\in Q_1$, $\varphi(b)=\varphi_2(b)\in Q_2$ and 
$$\varphi_1 f_1(a)=\varphi_1(a\overline{*}b)=\varphi(a\overline{*}b)=\varphi(a)\overline{*}\varphi(b)=\varphi_1(a)\overline{*}\varphi_2(b)=f_1\varphi_1(a),$$
i.~e. $\varphi_1\in C_{{\rm Aut}(Q_1)}(f_1)$. Similarly, we can show that $\varphi_2\in C_{{\rm Aut}(Q_2)}(f_2)$.

(2) Let $\alpha:Q_1\to Q_2$ be an isomorpism. Similar to case (1), we see that if $\varphi_1\in C_{{\rm Aut}(Q_1)}(f_1)$, $\varphi_2\in C_{{\rm Aut}(Q_2)}(f_2)$, then the map $\varphi=(\varphi_1,\varphi_2)$ given by $\varphi(x)=\varphi_i(x)$ for $x\in X_i$ is an automorphism of $B(Q_1~{}_{f_2}\bigsqcup{}_{f_1}~Q_2)$. The subgroup of ${\rm Aut}(B(Q_1~{}_{f_2}\bigsqcup{}_{f_1}~Q_2))$ generated by such automorphisms is isomorphic to $C_{{\rm Aut}(Q_1)}(f_1)\times C_{{\rm Aut}(Q_2)}(f_2)$. Let us prove that every automorphism $\varphi\in {\rm Aut}(B(Q_1~{}_{f_2}\bigsqcup{}_{f_1}~Q_2))$ can be written as $\varphi=(\varphi_1,\varphi_2)$ for $\varphi_1\in C_{{\rm Aut}(Q_1)}(f_1)$, $\varphi_2\in C_{{\rm Aut}(Q_2)}(f_2)$.

Let $\varphi$ be an automorphism of $B(Q_1~{}_{f_2}\bigsqcup{}_{f_1}~Q_2)$. Since $\varphi$ induces an automorphism of the quandle $\mathcal{Q}(B(Q_1~{}_{f_2}\bigsqcup{}_{f_1}~Q_2))=Q_1\sqcup Q_2$, from Lemma~\ref{unionaut}(2) follows that either $\varphi=(\varphi_1,\varphi_2)$ or $\varphi=\iota(\varphi_1,\varphi_2)$, where $\varphi_1\in {\rm Aut}(Q_1)$, $\varphi_2\in {\rm Aut}(Q_2)$ and $\iota:Q_1\sqcup Q_2\to Q_1\sqcup Q_2$ is given by formula~(\ref{iota}). Let us prove that the map of the form $\varphi=\iota(\varphi_1,\varphi_2)$ cannot be an automorphism of $B(Q_1~{}_{f_2}\bigsqcup{}_{f_1}~Q_2)$ for any $\varphi_1\in {\rm Aut}(Q_1)$, $\varphi_2\in {\rm Aut}(Q_2)$. On the contrary, suppose that $\varphi=\iota(\varphi_1,\varphi_2)$ is an automorphism of $B(Q_1~{}_{f_2}\bigsqcup{}_{f_1}~Q_2)$. Then for $a\in Q_1$, $b\in Q_2$, we have
\begin{align}
\notag \varphi(a\overline{*}b)&=\varphi(f_1(a))=\iota(\varphi_1,\varphi_2)(f_1(a))=\iota\varphi_1f_1(a)=\alpha\varphi_1f_1(a),\\
\varphi(a)\overline{*}\varphi(b)&=f_2\varphi(a)=f_2\iota(\varphi_1,\varphi_2)(a)=f_2\iota\varphi_1(a)=f_2\alpha\varphi_1(a).
\end{align}
Thus,  $\varphi_1f_1\varphi_1^{-1}=\alpha^{-1}f_2\alpha$, i.~e. $f_1$ and $\alpha^{-1}f_2\alpha$ are conjugate in ${\rm Aut}(Q_1)$ which contradicts the hypothesis. Hence, $\varphi=(\varphi_1,\varphi_2)$ for $\varphi_1\in {\rm Aut}(Q_1)$, $\varphi_2\in {\rm Aut}(Q_2)$ and similar to case (1) we conclude that $\varphi_1\in C_{{\rm Aut}(Q_1)}(f_1)$, $\varphi_2\in C_{{\rm Aut}(Q_2)}(f_2)$.

(3) Let $\alpha:Q_1\to Q_2$ be an isomorphism, and $\alpha^{-1} f_2\alpha=\psi^{-1}f_1\psi$ for $\psi\in{\rm Aut}(Q_1)$. Denote by $\alpha_1=\alpha\psi^{-1}$. It is clear that $\alpha_1:Q_1\to Q_2$ is an isomorphism. Denote by
\begin{align}\label{iota1}
\iota_1(x)=\begin{cases}
\alpha_1(x),& x\in Q_1,\\
\alpha_1^{-1}(x),&x\in Q_2,
\end{cases}
\end{align}
and let us check that $\iota_1$ is an automorphism of $B(Q_1~{}_{f_2}\bigsqcup{}_{f_1}~Q_2)$. It is clear that $\iota_1$ is bijective, we just need to check that $\iota_1$ respects the operations $\overline{*}, \underline{*}$ of $B(Q_1~{}_{f_2}\bigsqcup{}_{f_1}~Q_2)$. If  $a,b\in Q_1$, then $\iota_1(a), \iota_1(b)\in Q_2$ and 
\begin{align}
\notag\iota_1(a\overline{*}b)&=\iota_1(a)=\iota_1(a)\overline{*}\iota_1(b),\\ 
\notag\iota_1(a\underline{*}b)&=\iota_1(a*_1b)=\alpha_1(a*_1b)=\alpha_1(a)*_2\alpha_1(b)=\iota_1(a)\overline{*}\iota_1(b).
\end{align}
In a similar way we can prove that if $a,b\in Q_2$, then $\iota_1(a\overline{*}b)=\iota_1(a)\overline{*}\iota_1(b)$, $\iota_1(a\underline{*}b)=\iota_1(a)\underline{*}\iota_1(b)$. If $a\in Q_1$, $b\in Q_2$, then $\iota_1(a)\in Q_2$, $\iota_1(b)\in Q_1$ and 
\begin{align}
\notag\iota_1(a\overline{*}b)&=\iota_1(f_1(a))=\alpha_1f_1(a)=\alpha\psi^{-1}f_1(a)=\alpha\alpha^{-1}f_2\alpha\psi^{-1}(a)\\
\notag&=f_2\alpha_1(a)=\alpha_1(a)\overline{*}\alpha_1^{-1}(b)=\iota_1(a)\overline{*}\iota_1(b).
\end{align} 
Since for $a\in Q_1$, $b\in Q_2$, we have $a\overline{*}b=a\underline{*}b$, we conclude that if $a\in Q_1$, $b\in Q_2$, then $\iota_1(a\underline{*}b)=\iota_1(a)\underline{*}\iota_1(b)$. In a similar way we can prove that if  $a\in Q_2$, $b\in Q_1$, then $\iota_1(a\overline{*}b)=\iota_1(a)\overline{*}\iota_1(b)$, $\iota_1(a\underline{*}b)=\iota_1(a)\underline{*}\iota_1(b)$. Thus, $\iota_1$ is an automorphism of $B(Q_1~{}_{f_2}\bigsqcup{}_{f_1}~Q_2)$. It is clear that $\iota_1$ is of  order $2$.

Let us prove now that every automorphism $\varphi$ of $B(Q_1~{}_{f_2}\bigsqcup{}_{f_1}~Q_2)$ can be written as either $\varphi=(\varphi_1,\varphi_2)$ or $\varphi=\iota_1(\varphi_1,\varphi_2)$ for $\varphi_1\in C_{{\rm Aut}(Q_1)}(f_1)$, $\varphi_2\in C_{{\rm Aut}(Q_2)}(f_2)$. 
Let $\varphi$ be an automorphism of $B(Q_1~{}_{f_2}\bigsqcup{}_{f_1}~Q_2)$. Since $\varphi$ induces an automorphism of $\mathcal{Q}(B(Q_1~{}_{f_2}\bigsqcup{}_{f_1}~Q_2))=Q_1\sqcup Q_2$, from Lemma~\ref{unionaut}(2) follows that $\varphi$ is either $\varphi=(\varphi_1,\varphi_2)$ or $\varphi=\iota_1(\varphi_1,\varphi_2)$, where $\varphi_1\in {\rm Aut}(Q_1)$, $\varphi_2\in {\rm Aut}(Q_2)$ and $\iota_1:Q_1\sqcup Q_2\to Q_1\sqcup Q_2$ is given by formula~(\ref{iota1}). If $\varphi=(\varphi_1,\varphi_2)$, then similar to case (1) it is easy to see that $\varphi_1\in C_{{\rm Aut}(Q_1)}(f_1)$, $\varphi_2\in C_{{\rm Aut}(Q_2)}(f_2)$. If $\varphi=\iota_1(\varphi_1,\varphi_2)$, then for $a\in Q_1$, $b\in Q_2$, we have 
\begin{align}
\notag\varphi(a\overline{*}b)&=\varphi(f_1(a))=\iota_1(\varphi_1,\varphi_2)(f_1(a))=\iota_1\varphi_1f_1(a)=\alpha_1\varphi_1f_1(a),\\
\notag\varphi(a)\overline{*}\varphi(b)&=f_2\varphi(a)=f_2\iota_1(\varphi_1,\varphi_2)(a)=f_2\iota_1\varphi_1(a)=f_2\alpha_1\varphi_1(a).
\end{align}
Therefore $\alpha_1\varphi_1f_1=f_2\alpha_1\varphi_1$ (since $\varphi$ is an automorphism of $B(Q_1~{}_{f_2}\bigsqcup{}_{f_1}~Q_2)$). From this equality follows that 
$$f_1=\varphi_1^{-1}\alpha_1^{-1}f_2\alpha_1\varphi_1=\varphi_1^{-1}\psi\alpha^{-1}f_2\alpha\psi^{-1}\varphi_1=\varphi_1^{-1}f_1\varphi_1,$$
i.~e. $\varphi_1\in C_{{\rm Aut}(Q_1)}(f_1)$. In a similar way we can prove that $\varphi_2\in C_{{\rm Aut}(Q_2)}(f_2)$. Therefore the group 
${\rm Aut}(B(Q_1~{}_{f_2}\bigsqcup{}_{f_1}~Q_2))$ is generated by $C_{{\rm Aut}(Q_1)}(f_1)\times C_{{\rm Aut}(Q_2)}(f_2)$ and automorphism $\iota_1$ (of order~$2$). The equality
\begin{align}
\notag\iota_1(\varphi_1,\varphi_2)\iota_1(x)&=\begin{cases}
\alpha_1^{-1}\varphi_2\alpha_1(x),&x\in Q_1\\
\alpha_1 \varphi_1\alpha_1^{-1}(x),&x\in Q_2
\end{cases}\\
\notag&=\begin{cases}
(\alpha\psi^{-1})^{-1}\varphi_2(\alpha\psi^{-1})(x),&x\in Q_1\\
(\alpha\psi^{-1}) \varphi_1(\alpha\psi^{-1})^{-1}(x),&x\in Q_2
\end{cases}
\end{align}
implies that ${\rm Aut}(B(Q_1~{}_{f_2}\bigsqcup{}_{f_1}~Q_2))=(C_{{\rm Aut}(Q_1)}(f_1)\times C_{{\rm Aut}(Q_2)}(f_2))\rtimes\langle\iota_1\rangle$.
\end{proof}

\subsection{Automorphisms of product biquandles}\label{sec-product-biquandle}
 The following result about automorphisms of product biquandles is proved in \cite[Proposition~5.3]{Horvat}.
\begin{proposition}\label{classicalhorvataut}Let $Q_1,Q_2$ be connected quandles. Then ${\rm Aut}(B(Q_1\times Q_2))={\rm Aut}(Q_1)\times {\rm Aut}(Q_2)$.
\end{proposition}
In this section we study automorphisms of the product biquandles of the form $B(Q_1\times_{\psi}Q_2)$. Such quandles generalize quandles of the form $B(Q_1\times Q_2)$, and the results of this section essentially extend the result of Proposition~\ref{classicalhorvataut}.

Let $Q_1=(X_1, *_1)$, $Q_2=(X_2, *_2)$ be quandles, and $\psi: Q_2 \to \Conj_{-1} (\Aut (Q_1))$ be a homomorphism. For $f\in X_2$ denote by $\psi(f)=\widehat{f}$. Then the set $X_1\times X_2$ with the operations
\begin{align}
\notag(x, f) \underline{*} (y, g)&= \big(\widehat{g}(x*_1y), f \big)\\ \notag(x, f) \overline{*} (y, g)&= \big(\widehat{g}(x), f *_2 g\big)
\end{align}
for $(x, f), (y, g) \in Q_1 \times Q_2$ is a biquandle (see Section~\ref{newproductofquandles}) which is denoted by $B(Q_1\times_{\psi}Q_2)$. 
\begin{lemma}\label{compat}Let $Q_1$, $Q_2$ be quandles, and $\psi: Q_2 \to \Conj_{-1} (\Aut (Q_1))$ be a homomorphism. Then the set ${\rm Aut}_{\psi}(Q_1\times Q_2)=\{(\alpha,\beta)~|~\alpha\in {\rm Aut}(Q_1), \beta\in {\rm Aut}(Q_2), \widehat{\beta(f)}=\alpha\widehat{f}\alpha^{-1}~\text{for all}~f\in Q_2\}$ is a subgroup of ${\rm Aut}(Q_1)\times {\rm Aut}(Q_2)$.
\end{lemma}
\begin{proof} A direct check.
\end{proof}

\begin{proposition}\label{auto-product-subgroup}
Let $Q_1$, $Q_2$ be quandles, and $\psi: Q_2 \to \Conj_{-1} (\Aut (Q_1))$ be a homomorphism. Let $Q_2$ has orbits $P_1,P_2,\dots,P_k$, and $\chi:Q_2\to \{1,\dots,k\}$ be a map given by $\chi(x)=i$ if $x\in P_i$. Let $\delta_1, \delta_2,\dots,\delta_k\in C_{{\rm Aut}(Q_1)}(\psi(Q_2)\cup{\rm Inn}(Q_1))$, and $(\alpha,\beta)\in{\rm Aut}_{\psi}(Q_1\times Q_2)$. Then the map $\varphi$ given by 
\begin{equation}\label{specialform}
\varphi(x,f)=(\alpha\delta_{\chi(f)}(x),\beta(f))
\end{equation}
is an automorphism of $B(Q_1\times_{\psi}Q_2)$. The set of all such automorphisms forms a subgroup in ${\rm Aut}(B(Q_1\times_{\psi}Q_2))$.
\end{proposition}

\begin{proof}Let us prove that $\varphi$ is bijective. If for $x,y\in Q_1$, $f,g\in Q_2$ we have
$$(\alpha\delta_{\chi(f)}(x),\beta(f))=\varphi(x,f)=\varphi(y,g)=(\alpha\delta_{\chi(g)}(y),\beta(g)),$$
then from the second coordinate we obtain $\beta(f)=\beta(g)$, and therefore $f=g$. Comparing the first coordinate gives $x=y$. Thus, we proved that $\varphi$ is injective. Let $(y,g)\in B(Q_1\times_{\psi}Q_2)$. Since $\beta$ is an automorphism of $Q_2$, there exists $f\in Q_2$ such that $\beta(f)=g$. If we denote by $x=\delta_{\chi(f)}^{-1}\alpha^{-1}(y)$, then 
$$\varphi(x,f)=(\alpha\delta_{\chi(f)}(x),\beta(f))=(y,g),$$
Hence, we proved that $\varphi$ is bijective. Let us prove that $\varphi$ is a homomorphism. For $x,y\in Q_1$, $f,g\in Q_2$ we have
\begin{align}
\label{nec11} \varphi((x,f)\underline{*}(y,g))&=\varphi(\widehat{g}(x*_1y),f)=(\alpha\delta_{\chi(f)}\widehat{g}(x*_1y),\beta(f)),\\
\label{nec12}\varphi(x,f)\underline{*}\varphi(y,g)&=(\alpha\delta_{\chi(f)}(x),\beta(f))\underline{*}(\alpha\delta_{\chi(g)}(y),\beta(g))\\
\notag&=(\widehat{\beta(g)}(\alpha\delta_{\chi(f)}(x)*_1\alpha\delta_{\chi(g)}(y)), \beta(f))\\
\notag&=(\widehat{\beta(g)}\alpha(\delta_{\chi(f)}(x)*_1\delta_{\chi(g)}(y)), \beta(f)).
\end{align}
Since $\delta_i\in C_{{\rm Aut}(Q_1)}({\rm Inn}(Q_1))$ for $i=1,\dots,k$, for $y\in Q_1$, we have 
\begin{equation}\label{simpleconj}
\delta_iS_y\delta_i^{-1}=S_y,
\end{equation} 
where $S_y(x)=x*_1y$ is an inner automorphism. Further, since 
$$
\delta_iS_y\delta_i^{-1}(x)=\delta_i(\delta_i^{-1}(x)*_1y)=x*_1\delta_i(y),
$$
from equality (\ref{simpleconj}) for all $x,y\in Q_1$, $i,j=1,\dots,k$ we have $x*_1\delta_i(x)=x*_1y=x*_1\delta_j(y)$, in particular,
\begin{equation}\label{verytriv}
\delta_{\chi(f)}(x)*_1\delta_{\chi(g)}(x)=\delta_{\chi(f)}(x)*_1\delta_{\chi(f)}(y)=\delta_{\chi(f)}(x*_1y).
\end{equation} 
Since $(\alpha,\beta)\in {\rm Aut}_{\psi}(Q_1\times Q_2)$, we have the equality  $\widehat{\beta(g)}=\alpha\widehat{g}\alpha^{-1}$. Therefore
\begin{align}\label{nec13}\widehat{\beta(g)}\alpha(\delta_{\chi(f)}(x)*_1\delta_{\chi(g)}(y))&=\alpha\widehat{g}\delta_{\chi(f)}(x*_1y),\text{using}~(\ref{verytriv})\\
\notag&=\alpha\delta_{\chi(f)}\widehat{g}(x*_1y),\text{since}~\delta_{\chi(f)}\in C_{{\rm Aut}(Q_1)}(\psi(Q_2))
\end{align}
Equalities (\ref{nec11}), (\ref{nec12}) and  (\ref{nec13}) imply that $\varphi((x,f)\underline{*}(y,g))=\varphi(x,f)\underline{*}\varphi(y,g)$, and hence the map $\varphi$ respects the operation $\underline{*}$. Let us prove that $\varphi$ respects the operation $\overline{*}$. For $x,y\in Q_1$, $f,g\in Q_2$, we have
\begin{align}
\label{nec21} \varphi((x,f)\overline{*}(y,g))&=\varphi(\widehat{g}(x),f*_2g)=(\alpha\delta_{\chi(f*_2g)}\widehat{g}(x),\beta(f*_2g))\\
\notag &=(\alpha\delta_{\chi(f)}\widehat{g}(x),\beta(f*_2g)),~\text{since}~f~\text{and}~f*_2g~\text{are in the same orbit}\\
\label{nec22}\varphi(x,f)\overline{*}\varphi(y,g)&=(\alpha\delta_{\chi(f)}(x),\beta(f))\overline{*}(\alpha\delta_{\chi(g)}(y),\beta(g))=(\widehat{\beta(g)}\alpha\delta_{\chi(f)}(x),\beta(f)*_2\beta(g))\\
\notag&=(\alpha\widehat{g}\delta_{\chi(f)}(x),\beta(f*_2g)),~\text{since}~(\alpha,\beta)\in{\rm Aut}_{\psi}(Q_1\times Q_2).
\end{align}
Since $\delta_{\chi(f)}\in C_{{\rm Aut}(Q_1)}(\psi(Q_2))$, equalities (\ref{nec21}), (\ref{nec22}) imply $\varphi((x,f)\underline{*}(y,g))=\varphi(x,f)\underline{*}\varphi(y,g)$, i.~e. $\varphi$ respects the operation $\overline{*}$, and therefore $\varphi$ is an automorphism of $B(Q_1\times_{\psi}Q_2)$. 
\end{proof}
Denote by $H$ the subgroup of ${\rm Aut}(B(Q_1\times_{\psi}Q_2))$ generated by all automorphisms of the form (\ref{specialform}) and let us understand how this subgroup $H$ looks like. 
\begin{lemma}\label{easysubgroups} ${\rm Aut}_{\psi}(Q_1\times Q_2)\leq H$, $\left(C_{{\rm Aut}(Q_1)}\left(\psi(Q_2)\cup {\rm Inn}(Q_1)\right)\right)^k\trianglelefteq H$.
\end{lemma}
\begin{proof}The set of automorphisms of $B(Q_1\times_{\psi}Q_2)$ of the form
$$(x,f)\mapsto (\alpha(x),\beta(f))$$
for $(\alpha,\beta)\in {\rm Aut}_{\psi}(Q_1\times Q_2)$ clearly forms a subgroup of $H$  isomorphic to  ${\rm Aut}_{\psi}(Q_1\times Q_2)$.

For $\delta_1,\delta_2,\dots,\delta_k\in C_{{\rm Aut}(Q_1)}\left(\psi(Q_2)\cup {\rm Inn}(Q_1)\right)$ denote by $(\delta_1,\delta_2,\dots,\delta_k)$ the automorphism of $B(Q_1\times_{\psi}Q_2)$ which acts by the rule
\begin{equation}\label{directproductaut}(\delta_1,\dots,\delta_k)(x,f)=(\delta_{\chi(f)}(x),f).
\end{equation}
Since for all $\delta_1,\dots,\delta_k, \xi_1,\dots,\xi_k\in C_{{\rm Aut}(Q_1)}(\psi(Q_2)\cup {\rm Inn}(Q_1))$ the equality
$$(\xi_1,\dots,\xi_k)(\delta_1,\dots,\delta_k)(x,f)=(\xi_1,\dots,\xi_k)(\delta_{\chi(f)}(x),f)=(\xi_{\chi(f)}\delta_{\chi(f)}(x),f),$$
holds, we conclude that $(\xi_1,\dots,\xi_k)(\delta_1,\dots,\delta_k)=(\xi_1\delta_1,\dots,\xi_k\delta_k)$, i.~e. the subgroup of $H$ generated by all automorphisms of the form (\ref{directproductaut}) is isomorphic to $\left(C_{{\rm Aut}(Q_1)}\left(\psi(Q_2)\cup {\rm Inn}(Q_1)\right)\right)^k$. 

Let $\varphi$ be an automorphism from $H$ of the form 
$$\varphi(x,f)=(\alpha\delta_{\chi(f)}(x)),\beta(f)$$
for $(\alpha,\beta)\in {\rm Aut}_{\psi}(Q_1\times Q_2)$, $\delta_1,\dots,\delta_k\in C_{{\rm Aut}(Q_1)}(\psi(Q_2)\cup {\rm Inn}(Q_1))$. Then
$$\varphi^{-1}(x,f)=\left(\delta^{-1}_{\chi(\beta^{-1}(f))}\alpha^{-1}(x),\beta^{-1}(f)\right).$$
For $\psi=(\xi_1,\xi_2,\dots,\xi_k)\in \left(C_{{\rm Aut}(Q_1)}\left(\psi(Q_2)\cup {\rm Inn}(Q_1)\right)\right)^k$ we have
\begin{align}
\label{conjugationinH}\varphi\psi\varphi^{-1}(x,f)&=\varphi\psi\left(\delta^{-1}_{\chi(\beta^{-1}(f))}\alpha^{-1}(x),\beta^{-1}(f)\right)\\
\notag&=\varphi\left(\xi_{\chi(\beta^{-1}(f))}\delta^{-1}_{\chi(\beta^{-1}(f))}\alpha^{-1}(x),\beta^{-1}(f)\right)\\
\notag&=\left(\alpha\left(\delta_{\chi(\beta^{-1}(f))}\xi_{\chi(\beta^{-1}(f))}\delta^{-1}_{\chi(\beta^{-1}(f))}\right)\alpha^{-1}(x),f\right).
\end{align}
Since $\alpha\left(\delta_{\chi(\beta^{-1}(f))}\xi_{\chi(\beta^{-1}(f))}\delta_{\chi(\beta^{-1}(f))}^{-1}\right)\alpha^{-1}$ belongs to $C_{{\rm Aut}(Q_1)}\left(\psi(Q_2)\cup {\rm Inn}(Q_1)\right)$, we conclude that $\varphi\psi\varphi^{-1}$ belongs to $\left(C_{{\rm Aut}(Q_1)}\left(\psi(Q_2)\cup {\rm Inn}(Q_1)\right)\right)^k$.
\end{proof}
\begin{proposition}\label{whenorbitfixed}If $Q_2$ is a union of orbits $P_1,P_2,\dots,P_k$, and $\beta(P_1)=P_1$ for all $\beta\in {\rm Aut}(Q_2)$, then $H=\left(C_{{\rm Aut}(Q_1)}\left(\psi(Q_2)\cup {\rm Inn}(Q_1)\right)\right)^{k-1}\rtimes {\rm Aut}_{\psi}(Q_1\times Q_2)$.
\end{proposition}
\begin{proof}From  Lemma~\ref{easysubgroups} we see that the set of automorphisms $(\delta_1,\delta_2,\delta_3,\dots,\delta_k)$ of the form
\begin{equation}\label{directproductaut2}(\delta_1,\dots,\delta_k)(x,f)=(\delta_{\chi(f)}(x),f)
\end{equation}
with $\delta_1=id$ form a subgroup in $H$ isomorphic to $\left(C_{{\rm Aut}(Q_1)}\left(\psi(Q_2)\cup {\rm Inn}(Q_1)\right)\right)^{k-1}$. Denote this subgroup by $A$. Let $\varphi$ be an automorphism from $H$ of the form 
$$\varphi(x,f)=\left(\alpha\delta_{\chi(f)}(x),\beta(f)\right),$$
and $\psi=(\xi_1=id, \xi_2,\dots,\xi_k)$. From (\ref{conjugationinH}) follows that 
\begin{equation}\label{Aisnorm}\varphi\psi\varphi^{-1}(x,f)=\left(\alpha\left(\delta_{\chi(\beta^{-1}(f))}\xi_{\chi(\beta^{-1}(f))}\delta^{-1}_{\chi(\beta^{-1}(f))}\right)\alpha^{-1}(x),f\right)=(\eta_{\chi(f)}(x),f),
\end{equation}
where $\eta_{\chi(f)}=\alpha\left(\delta_{\chi(\beta^{-1}(f))}\xi_{\chi(\beta^{-1}(f))}\delta^{-1}_{\chi(\beta^{-1}(f))}\right)\alpha^{-1}$. If $f\in P_1$, then $\chi(f)=1$, and since $\beta(P_1)=P_1$, we have $\chi(\beta^{-1}(f))=\chi(f)=1$. Thus,
$$\eta_1=\alpha\left(\delta_{1}\xi_{1}\delta^{-1}_{1}\right)\alpha^{-1}=\alpha\left(\delta_{1}(id)\delta^{-1}_{1}\right)\alpha^{-1}=id,$$
and from (\ref{Aisnorm}) follows that $A$ is normal in $H$.

Denote by $B$ the set of automorphisms of $B(Q_1\times_{\psi}Q_2)$ of the form
$$(x,f)\mapsto (\alpha(x),\beta(f))$$
for $(\alpha,\beta)\in {\rm Aut}_{\psi}(Q_1\times Q_2)$. By Lemma~\ref{easysubgroups}, $B$ is a subgroup of $H$  isomorphic to  ${\rm Aut}_{\psi}(Q_1\times Q_2)$. It is clear that $A\cap B=1$. 

If $\varphi$ is an automorphism from $H$ of the form 
$$\varphi(x,f)=(\alpha\delta_{\chi(f)}(x)),\beta(f),$$
for $\delta_1,\delta_2,\dots,\delta_k\in C_{{\rm Aut}(Q_1)}\left(\psi(Q_2)\cup {\rm Inn}(Q_1)\right)$, then $\varphi=\psi_1\psi_2$, where $\psi_1$ is an automorphism from $B$ of the form
$$\psi_1(x,f)=(\alpha\delta_1(x),\beta(f)),$$
and $\psi_2=(id,\delta_1^{-1}\delta_2,\delta_1^{-1}\delta_3,\dots, \delta_1^{-1}\delta_k)$ is an automorphism from $A$. Thus, we obtain $H=AB=A\rtimes B$, and the proof is complete.
\end{proof}

\begin{corollary}\label{subgconnected}If $Q_2$ is connected, then $H={\rm Aut}_{\psi}(Q_1\times Q_2)$.
\end{corollary}

\begin{proof}Since $Q_2$ is connected, it has only one orbit (i.~e. $k=1$) which is, of course, fixed by all automorphisms. From Proposition~\ref{whenorbitfixed} follows that $H={\rm Aut}_{\psi}(Q_1\times Q_2)$. 
\end{proof}

\begin{proposition}\label{shsequence}
There is a short exact sequence of groups
$$1\to C_{{\rm Aut}(Q_1)}\left(\psi(Q_2)\cup {\rm Inn}(Q_1)\right)\to \left(C_{{\rm Aut}(Q_1)}\left(\psi(Q_2)\cup {\rm Inn}(Q_1)\right)\right)^k\rtimes {\rm Aut}_{\psi}(Q_1\times Q_2)\to H\to 1.$$
\end{proposition}

\begin{proof}By Lemma~\ref{easysubgroups} we can think of $\left(C_{{\rm Aut}(Q_1)}\left(\psi(Q_2)\cup {\rm Inn}(Q_1)\right)\right)^k$ as  the set of automorphisms of $B(Q_1\times_{\psi}Q_2)$ of the form (\ref{directproductaut}). If $\delta\in C_{{\rm Aut}(Q_1)}\left(\psi(Q_2)\cup {\rm Inn}(Q_1)\right)$, and $(\alpha,\beta)\in {\rm Aut}_{\psi}(Q_1\times Q_2)$, then a direct check shows that $\alpha\delta\alpha^{-1}\in C_{{\rm Aut}(Q_1)}\left(\psi(Q_2)\cup {\rm Inn}(Q_1)\right)$. For $(\alpha,\beta)$ denote by $\theta_{(\alpha,\beta)}$ the map from $\left(C_{{\rm Aut}(Q_1)}\left(\psi(Q_2)\cup {\rm Inn}(Q_1)\right)\right)^k$ to itself given by the following rule: if $\varphi=(\delta_1,\dots,\delta_k)$ is given by (\ref{directproductaut}), then $\theta_{(\alpha,\beta)}(\varphi)$ has the form
$$\theta_{(\alpha,\beta)}(\varphi)(x,f)=(\alpha\delta_{\beta(f)}\alpha^{-1},f).$$ 
Using direct check it is easy to see that the map $\theta:(\alpha,\beta)\mapsto\theta_{(\alpha,\beta)}$ is a homomorphism of groups
$$\theta:{\rm Aut}_{\psi}(Q_1\times Q_2)\to {\rm Aut}\left(C_{{\rm Aut}(Q_1)}\left(\psi(Q_2)\cup {\rm Inn}(Q_1)\right)\right)^k,$$
and we can define the semidirect product $\left(C_{{\rm Aut}(Q_1)}\left(\psi(Q_2)\cup {\rm Inn}(Q_1)\right)\right)^k\rtimes_{\theta} {\rm Aut}_{\psi}(Q_1\times Q_2)$. Denote by $\eta:\left(C_{{\rm Aut}(Q_1)}\left(\psi(Q_2)\cup {\rm Inn}(Q_1)\right)\right)^k\rtimes_{\theta} {\rm Aut}_{\psi}(Q_1\times Q_2)\to H$ the map given by 
$$\eta((\delta_1,\dots,\delta_n),(\alpha,\beta))=\varphi,$$ 
where $\varphi$ is an automorphism from $H$ of the form 
$$\varphi(x,f)=(\alpha\delta_{\chi(f)}(x),\beta(f)).$$
Using direct check it is easy to see that $\eta$ is a homomorphism. The kernel of this homomorphism is the set of elements $((\delta_1,\dots,\delta_n),(\alpha,\beta))$ such that 
$$(\alpha\delta_{\chi(f)}(x),\beta(f))=(x,f)$$
for all $x\in Q_1$, $f\in Q_2$. It means that $\beta=id$, $\alpha\in C_{{\rm Aut}(Q_1)}(\psi(Q_2)\cup {\rm Inn}(Q_1))$, and $\delta_i=\alpha^{-1}$ for all $i=1,\dots,k$. Thus,
$${\rm Ker}(\theta)=\{((\delta,\dots,\delta),~(\delta^{-1},id))~|~\delta\in C_{{\rm Aut}(Q_1)}(\psi(Q_2)\cup {\rm Inn}(Q_1))\}=C_{{\rm Aut}(Q_1)}(\psi(Q_2)\cup {\rm Inn}(Q_1)),$$
which completes the proof.
\end{proof}
\begin{corollary}\label{faithfulaut}If $Q_1$ is faithful, then $H={\rm Aut}_{\psi}(Q_1\times Q_2)$.
\end{corollary}
\begin{proof}If $\delta\in C_{{\rm Aut}(Q_1)}({\rm Inn}(Q_1))$, then for $x\in Q_1$ we have $S_x=\delta S_x\delta^{-1}=S_{\delta(x)}$. Since $Q_1$ is faithful, we get $\delta(x)=x$. Further, since $x$ is an arbitrary element of $Q_1$, we conclude that $\delta=id$. Therefore $C_{{\rm Aut}(Q_1)}({\rm Inn}(Q_1))$ is trivial, and from Proposition~\ref{shsequence} follows that $H={\rm Aut}_{\psi}(Q_1\times Q_2)$.
\end{proof}

The following result says that in some cases the group $H$ coincides with the whole group ${\rm Aut}(B(Q_1\times_{\psi}Q_2))$.

\begin{theorem}\label{automconprod}
Let $Q_1$, $Q_2$ be finite quandles, such that $Q_1$ is connected, and $id\in \psi(Q_2)$. Then $H={\rm Aut}(B(Q_1\times_{\psi}Q_2))$.
\end{theorem}
\begin{proof}Let $h$ be an element from $Q_2$ such that $\widehat{h}=id$ (it exists since $id\in\psi(Q_2)$), and let $\varphi$ be an arbitrary automorphism of ${\rm Aut}(B(Q_1\times_{\psi}Q_2))$. Let us prove that $\varphi$ belongs to $H$.

Let $x\in Q_1$, $f\in Q_2$, and $\varphi(x,f)=(y,g)$ for some $y\in Q_1$, $g\in Q_2$. Since $Q_1$  is connected, for every $z\in Q_1$ there  exist some elements $x_1,x_2,\dots,x_n$ such that $z=x*_1x_1*_1x_2*_1\dots*_1x_n$. Since $\widehat{h}=id$, we have
$$(x,f)\underline{*}(x_1,h)=(\widehat{h}(x*_1x_1),f)=(x*_1x,f).$$
In a similar way it is easy to see that
$$(x,f)\underline{*}(x_1,h)\underline{*}(x_2,h)\underline{*}\cdots\underline{*}(x_n,h)=(x*_1x_1*_1x_2*_1\dots*_1x_n,f)=(z,f).$$
Acting to this equality by $\varphi$, we have 
\begin{equation}\label{reduceto}
\varphi(z,f)=(y,g)\underline{*}(y_1,g_1)\underline{*}(y_2,g_2)\underline{*}\dots\underline{*}(y_n,g_n),
\end{equation}
where  $(y_i,g_i)=\varphi(x_i,h)$. By definition of the operation $\underline{*}$, we see that equality (\ref{reduceto}) implies that 
$\varphi(z,f)=(t,g)$ for some $t\in Q_1$. Thus, the second coordinate of $\varphi(z,f)$ depends only on the second coordinate of the argument, i.~e. we can write 
\begin{equation}\label{startingpoint}
\varphi(x,f)=(\alpha_f(x),\beta(f)).
\end{equation}
Let us prove that $\alpha_f\in {\rm Aut}(Q_1)$ for all $f\in Q_2$, and $\beta\in {\rm Aut}(Q_2)$. At first, let us prove that $\beta\in {\rm Aut}(Q_2)$. For arbitrary $x,y\in Q_1$, $f,g\in Q_2$, we have
\begin{align}
\notag \varphi((x,f)\overline{*}(y,g))&=\varphi(\widehat{g}(x),f*_2g)=(\alpha_{f*_2g}\widehat{g}(x),\beta(f*_2g)),\\
\notag\varphi(x,f)\overline{*}\varphi(y,g)&=(\alpha_f(x),\beta(f))\overline{*}(\alpha_g(y),\beta(g))=(\widehat{\beta(g)}\alpha_f(x),\beta(f)*_2\beta(g)).
\end{align}
Therefore $\beta(f*_2g)=\beta(f)*_2\beta(g)$, i.~e. $\beta$ is an endomorphism of $Q_2$. Let $y\in Q_1$, $g\in Q_2$ be arbitrary elements. Since $\varphi$ is an automorphism of $B(Q_1\times_{\psi}Q_2)$, there exists $(x,f)$ such that $(\alpha_f(x),\beta(f))=\varphi(x,f)=(y,g)$. Thus, $\beta(f)=g$, i.~e. $\beta$ is surjective, and since $Q_2$ is finite, $\beta:Q_2\to Q_2$ is a bijective homomorphism, i.~e. an automorphism of $Q_2$.

Let us prove that $\alpha_f\in{\rm Aut}(Q_1)$ for all $f\in Q_2$. Let $x,y\in Q_1$, then 
\begin{align}
\notag\varphi((x,f)\underline{*}(y,f))&=\varphi(\widehat{f}(x*_1y),f)=(\alpha_f\widehat{f}(x*_1y),\beta(f))\\
\notag\varphi(x,f)\underline{*}\varphi(y,f)&=(\alpha_f(x),\beta(f))\underline{*}(\alpha_f(y),\beta(f))=(\widehat{\beta(f)}(\alpha_f(x)*_1\alpha_f(y)),\beta(f))
\end{align}
Thus, for all $x,y\in Q_1$, we have 
\begin{equation}\label{lem1equ1}
\alpha_f\widehat{f}(x*_1y)=\widehat{\beta(f)}\left(\alpha_f(x)*_1\alpha_f(y)\right).
\end{equation}
From the other side 
\begin{align}
\notag\varphi((x,f)\overline{*}(y,f))&=\varphi(\widehat{f}(x),f)=(\alpha_f\widehat{f}(x),\beta(f))\\
\notag\varphi(x,f)\overline{*}\varphi(y,f)&=(\alpha_f(x),\beta(f))\underline{*}(\alpha_f(y),\beta(f))=(\widehat{\beta(f)}\alpha_f(x),\beta(f))
\end{align}
Thus, for all $x,y\in Q_1$, we have 
\begin{equation}\label{lem1equ2}
\alpha_f\widehat{f}(x)=\widehat{\beta(f)}\alpha_f(x)
\end{equation}
From equalities (\ref{lem1equ1}) and (\ref{lem1equ2}) follows that 
$$\widehat{\beta(f)}\alpha_f(x*_1y)=\widehat{\beta(f)}\left(\alpha_f(x)*_1\alpha_f(y)\right).
$$
Since $\widehat{\beta(f)}$ is an automorphism of $Q_1$, we have 
$$\alpha_f(x*_1y)=\alpha_f(x)*_1\alpha_f(y),$$
i.~e. $\alpha_f$ is an endomorphism of $Q_1$.  Let us prove that $\alpha_f$ is bijective. Since $Q_1$ is finite, it is enough to prove that $\alpha_f$ is injective. Let $x,y\in Q_1$ be such that $\alpha_f(x)=\alpha_f(y)$. Then 
$$
\varphi(x,f)=(\alpha_f(x),\beta(f))=(\alpha_f(y),\beta(f))=\varphi(y,f)$$
and since $\varphi$ is an automorphism of ${\rm Aut}(B(Q_1\times_{\psi}Q_2))$, it is injective, and therefore $x=y$. Thus, $\alpha_f$ is an automorphism of $Q_1$. From equality (\ref{lem1equ2}) follows that 
\begin{equation}\label{conjrule}
\widehat{\beta(f)}=\alpha_f\widehat{f}\alpha_f^{-1}.
\end{equation} 
Since $\varphi$ is an automorphism of $B(Q_1\times_{\psi}Q_2)$, for $x,y\in Q_1$, $f\in Q_2$, we have
\begin{align}\notag \varphi((x,f)\underline{*}(y,h))&=\varphi(x*_1y,f)=(\alpha_f(x*_1y),\beta(f)),\\
\notag \varphi(x,f)\underline{*}\varphi(y,h)&=(\alpha_f(x),\beta(f))\underline{*}(\alpha_{h}(y),\beta(h))=\left(\widehat{{\beta(h)}}(\alpha_f(x)*_1\alpha_{h}(y)),\beta(f)\right).
\end{align}
From this equality follows that for all $x,y\in Q_1$, $f\in Q_2$ there is an equality 
\begin{equation}\label{trivactth}
\alpha_f(x*_1y)=\widehat{{\beta(h)}}(\alpha_f(x)*_1\alpha_{h}(y)).
\end{equation}
From equality (\ref{conjrule}) follows that $\widehat{{\beta(h)}}=\alpha_{h}\widehat{h}\alpha_{h}^{-1}=\alpha_{h}(id)\alpha_{h}^{-1}=id$.  Therefore equality (\ref{trivactth}) implies that for all $x,y\in Q_1$, $f\in Q_2$ the equality
$\alpha_f(x)*_1\alpha_f(y)=\alpha_f(x)*_1\alpha_{h}(y)$ holds. Since $x,y$ are arbitrary, this equality can be rewritten as
\begin{equation}\label{trivactth2}
x*_1\alpha_f(y)=x*_1\alpha_{h}(y).
\end{equation}
For arbitrary $x,y\in Q_1$, $f\in Q_2$, we have
\begin{align}\notag \varphi((y,h)\underline{*}(x,f))&=\varphi(\widehat{f}(y*_1x),h)=(\alpha_{h}\widehat{f}(y*_1x),\beta(h)),\\
\notag \varphi(y,h)\underline{*}\varphi(x,f)&=(\alpha_{h}(y),\beta(h))\underline{*}(\alpha_f(x),\beta(f))=(\widehat{\beta(f)}(\alpha_{h}(y)*_1\alpha_f(x)),\beta(h))\\
\notag&=(\alpha_f\widehat{f}\alpha_f^{-1}\alpha_{h}(y*_1x),\beta(h)),~~\text{using}~(\ref{conjrule})~\text{and}~(\ref{trivactth2}).
\end{align}
From this equality follows that $\alpha_{h}\widehat{f}=\alpha_f\widehat{f}\alpha_f^{-1}\alpha_{h}$ or $\alpha_{h}\widehat{f}\alpha_{h}^{-1}=\alpha_f\widehat{f}\alpha_f^{-1}$. This equality implies that $\alpha_f=\alpha_{h}\delta_f$, where $\delta_f$ commutes with $\widehat{f}$. Denote by $\alpha=\alpha_h$. From equality (\ref{conjrule}) follows that 
$$\widehat{\beta(f)}=\alpha_f\widehat{f}\alpha_f^{-1}=\alpha\delta_f\widehat{f}\delta_f^{-1}\alpha^{-1}=\alpha\widehat{f}\alpha^{-1},$$
for all $f\in Q_2$, and hence $(\alpha,\beta)\in {\rm Aut}_{\psi}(Q_1\times Q_2)$. From equality (\ref{trivactth2}) follows that the equality 
\begin{equation}\label{trivactth4}
x*_1\alpha\delta_f(y)=x*_1\alpha(y)
\end{equation}
holds for all $x,y\in Q_1$. Since $x$ is an arbitrary element, from equality (\ref{trivactth4}) we conclude that 
$$\delta_fS_y\delta_f^{-1}(x)=S_{\delta_f(y)}(x)=x*_1\delta_f(y)=x*_1y=S_y(x),$$
i.~e. $\delta_fS_y\delta_f^{-1}=S_y$ and $\delta_f\in C_{{\rm Aut}(Q_1)}({\rm Inn}(Q_1))$. Summarizing the preceding discussion, we can rewrite (\ref{startingpoint}) in the following way
\begin{equation}\label{better2}
\varphi(x,f)=(\alpha\delta_f(x),\beta(f)),
\end{equation}
where $(\alpha,\beta)\in {\rm Aut}_{\psi}(Q_1\times Q_2)$, and hence $\delta_f\in C_{{\rm Aut}(Q_1)}({\rm Inn}(Q))$ is such that $\delta_f$ commutes with $\widehat{f}$.

For arbitrary $x,y\in Q_1$, $f,g\in Q_2$, we have
\begin{align}
\notag \varphi((x,f)\underline{*}(y,g))&=\varphi(\widehat{g}(x*_1y),f)=(\alpha\delta_f\widehat{g}(x*_1y),\beta(f)),\\
\notag \varphi(x,f)\underline{*}\varphi(y,g)&=(\alpha\delta_f(x),\beta(f))\underline{*}(\alpha\delta_g(y),\beta(g))=(\widehat{\beta(g)}(\alpha\delta_f(x)*_1\alpha\delta_g(y)),\beta(f))\\
\notag&=(\alpha\widehat{g}(\delta_f(x)*_1\delta_g(y)),\beta(f)),~\text{since}~(\alpha,\beta)\in {\rm Aut}_{\psi}(Q_1\times Q_2)\\
\notag&=(\alpha\widehat{g}(\delta_f(x)*_1\delta_f(y)),\beta(f)),~\text{since}~\delta_f\in C_{{\rm Aut}(Q_1)}({\rm Inn}(Q_1)).
\end{align}
These equalities imply that $\alpha\delta_f\widehat{g}=\alpha\widehat{g}\delta_f$, or $\delta_f\widehat{g}=\widehat{g}\delta_f$. Since $f$ and $g$ are arbitrary, we must have $\delta_f\in C_{{\rm Aut}(Q_1)}(\psi(Q_2))$. So, we conclude that $\delta_f\in C_{{\rm Aut}(Q_1)}({\rm Inn}(Q_1)\cup\psi(Q_2))$. 

Finally, for arbitrary $x,y\in Q_1$, $f,g\in Q_2$
\begin{align}
\notag \varphi((x,f)\overline{*}(y,g))&=\varphi(\widehat{g}(x),f*_2g)=\left(\alpha\delta_{f*_2g}\widehat{g}(x),\beta(f*_2g)\right),\\
\notag\varphi(x,f)\overline{*}\varphi(y,g)&=(\alpha\delta_f(x),\beta(f))\overline{*}(\alpha\delta_g(y),\beta(g))=(\widehat{\beta(g)}\alpha\delta_f(x),\beta(f)*_2\beta(g))\\
\notag&=(\alpha\widehat{g}\delta_f(x),\beta(f)*_2\beta(g)),~\text{since}~\delta_f\in C_{{\rm Aut}(Q_1)}({\rm Inn}(Q_1)).
\end{align}
From these equalities follows that $\delta_{f*_2g}=\delta_f$ for all $f,g\in Q_2$. Therefore $\delta_{f_1}=\delta_{f_2}$ for all $f_1,f_2$ which belong to the same orbit of $Q_2$, i.~e. $\delta_f$ depends not on $f$ but on the orbit of $f$, and we can write $\delta_{\chi(f)}$ instead of $\delta_f$. Thus, equality (\ref{better2}) can be now rewritten as 
$$
\varphi(x,f)=(\alpha\delta_{\chi(f)}(x),\beta(f)),
$$
where $(\alpha,\beta)\in {\rm Aut}_{\psi}(Q_1\times Q_2)$, and $\delta_{\chi(f)}\in C_{{\rm Aut}(Q_1)}({\rm Inn}(Q)\cup\psi(Q_2))$. 
\end{proof}
Theorem~\ref{automconprod} and Corollary~\ref{subgconnected} imply the following statement.
\begin{corollary}\label{genhorv1}Let $Q_1$, $Q_2$ be finite connected quandles, and $\psi:Q_2\to {\rm Conj}_{-1}({\rm Aut}(Q_1))$ be a homomorphism such that $id\in \psi(Q_2)$. Then ${\rm Aut}(B(Q_1\times_{\psi}Q_2))={\rm Aut}_{\psi}(Q_1\times Q_2)$.
\end{corollary}
In particular case when $\psi(Q_2)=\{id\}$, we have ${\rm Aut}_{\psi}(Q_1\times Q_2)={\rm Aut}(Q_1)\times {\rm Aut}(Q_2)$, and Corollary~\ref{genhorv1} imply the result \cite[Proposition~5.3]{Horvat} formulated in  Proposition~\ref{classicalhorvataut}. Theorem~\ref{automconprod} and Corollary~\ref{faithfulaut}  imply the following statement. 
\begin{corollary}\label{genhorv2}Let $Q_1$, $Q_2$ be finite quandles, and $\psi:Q_2\to {\rm Conj}_{-1}({\rm Aut}(Q_1))$ be a homomorphism such that $id\in\psi(Q_2)$. If $Q_1$ is a faithful connected quandle, then ${\rm Aut}(B(Q_1\times_{\psi}Q_2))={\rm Aut}_{\psi}(Q_1\times Q_2)$.
\end{corollary}
In particular case when $\psi(Q_2)=\{id\}$, Corollary~\ref{genhorv1} imply the following result, which generalizes the result \cite[Proposition 5.3]{Horvat} to the case when $Q_2$ is not necessarily connected.
\begin{corollary}\label{Horvatdidntit}Let $Q_1$, $Q_2$ be finite quandles, and $\psi(Q_2)=\{id\}$. If $Q_1$ is faithful and connected, then ${\rm Aut}(B(Q_1\times_{\psi}Q_2))={\rm Aut}(Q_1)\times {\rm Aut}(Q_2)$.
\end{corollary}
Note that the majority of finite connected quandles are faithful. For example, there are $790$ connected quandles in the GAP \cite{GAP} package Rig (http://code.google.com/p/rig/) written by L.~Vendramin (see also \cite{Ven}). These comprise all connected quandles of order at most $47$. Among these $790$ finite connected quandle of order at most $47$ only $66$ quandles are not faithful.

Another result which follows from Corollary~\ref{genhorv2} is regarding automorphisms of the holomorph of a finite connected quandle.

\begin{corollary}\label{authol}
Let $Q$ be a finite faithful connected quandle. Then ${\rm Aut}({\rm Hol}(Q))={\rm Aut}(Q)$.
\end{corollary}

\begin{proof}In this situation ${\rm Aut}_{\psi}(Q\times{\rm Conj}_{-1}({\rm Aut}(Q)))={\rm Aut}(Q)$.
\end{proof}

Corollary~\ref{authol} implies the following result which is a kind of analogue to \cite[Main Theorem]{BrayWil} for biquandles.

\begin{corollary}\label{seq-biquandles}
There exists a sequence of finite biquandles $B_1,B_2,\dots$ such that $|B_k|\to\infty$, and $|{\rm Aut}(B_k)|/|B_k|\to 0$.
\end{corollary}

\begin{proof}Let $B_k={\rm Hol}({\rm R}_{2k+1})$ be the holomorph of the dihedral quandle of order $2k+1$. It is clear that $|B_k|=(2k+1)|{\rm Aut}({\rm R}_{2k+1})|$, and therefore $|B_k|\to \infty$. Since the dihedral quandle ${\rm R}_{2k+1}$ is connected and faithful, from Corollary~\ref{authol} follows that ${\rm Aut}(B_k)={\rm Aut}({\rm R}_{2k+1})$. Therefore $|{\rm Aut}(B_k)|/|B_k|=1/(2k+1)\to 0$.
\end{proof}
If $G$ is a finite group of order at least $3$, then $G$ is known to have  a non-trivial automorphism. Similar results holds for quandles, namely, if $Q$ is a finite quandle of order at least $2$, then $Q$ has a non-trivial automorphism. These results follow from the fact that groups and quandles have families of natural automorphisms, namely, inner automorphisms. On the contrary, biquandles do not have any natural automorphisms. So, the following problem appears naturally. 
\begin{problem}Does there exist an integer $N$ such that every biquandle $B$ of order at least $N$ has a non-trivial automorphism?
\end{problem}
We believe that the problem formulated above has a negative solution, i.~e. that there exist a sequence of finite biquandles $B_1,B_2,\dots$ such that $|B_k|\to\infty$ and ${\rm Aut}(B_i)$ is trivial for each $i$. However, we could not find such a sequence of biquandles.

\subsection{Automorphisms of biquandles and coverings} Let $B$ be a biquandle and $p: \widetilde{Q} \to \mathcal{Q}(B)$ be a quandle covering. In this concluding section, we find some connections between 
automorphisms of the biquandle $B$ and automorphisms of the covering quandle~$\widetilde{Q}$.

\begin{proposition}\label{biquandle-auto-in-lifting}
Let $\widetilde{Q}$ be a simply conntected quandle, $p:\widetilde{Q}\to Q$ be a quandle covering, $A=\{\beta_y~|~y \in Q \}$ be a biquandle structure on $Q$, and $\widetilde{A}=\{\alpha_{\tilde{y}}~|~\tilde{y} \in \widetilde{X} \}$ be a biquandle structure on $\widetilde{Q}$ obtained from $A$ using Theorem~\ref{lifting-biquandle-structure}. If $B$ denote the biquandle obtained from $Q$ using the biquandle structure $A$, then
$${\rm Aut}(B) \leq N_{{\rm Aut}(\widetilde{Q})}\{\alpha_{\tilde{y}}~|~\tilde{y} \in \widetilde{Q} \}.$$
\end{proposition}

\begin{proof}
For $\tilde{y}\in \widetilde{Q}$ denote by $y=p(\tilde{y})$. By Theorem~\ref{lifting-biquandle-structure} for all $\tilde{y} \in \widetilde{Q}$ we have the equality $p \alpha_{\tilde{y}}= \beta_y p$. Let $\varphi$ be an automorphism from ${\rm Aut}(B)$.  By Theorem \ref{lifting theorem}, $\varphi$ lifts uniquely to $\widetilde{\varphi} \in {\rm Aut}(\widetilde{Q})$ such that $p \widetilde{\varphi}= \varphi p$, so, it is enough to show that $\widetilde{\varphi}$ belongs to $N_{{\rm Aut}(\widetilde{Q})}\{\alpha_{\tilde{y}}~|~\tilde{y} \in \widetilde{Q} \}$. Note that, by Theorem~\ref{biquandlestructureautomorphism}, $\varphi \in N_{{\rm Aut}(Q)} \{\beta_y~|~y \in Q\}$. If $\tilde{x} \in \widetilde{Q}$ and $x:= p(\tilde{x})$, then
\begin{align}
\notag p (\widetilde{\varphi} \alpha_{\tilde{x}}) &= (\varphi p) \alpha_{\tilde{x}}= \varphi (\beta_x p)\\
\notag&= (\beta_z \varphi) p,~\textrm{for some}~z \in Q\\
\notag&= \beta_z (p \widetilde{\varphi})= (p \alpha_{\tilde{z}}) \widetilde{\varphi}= p (\alpha_{\tilde{z}} \widetilde{\varphi}).
\end{align}
By the uniqueness of the lift, it follows that $\widetilde{\varphi} \alpha_{\tilde{x}}=\alpha_{\tilde{z}} \widetilde{\varphi}$, and the proof is complete.
\end{proof}

\begin{corollary}
Let $\widetilde{Q}$ be a simply connected quandle, $p: \widetilde{Q} \to Q$ be a quandle covering, $f$ be an automorphism of $Q$, and $\tilde{f} \in {\rm Aut}(\widetilde{Q})$ be the lift of $f$  with respect to $p$. Denote by $B,\widetilde{B}$ biquandles obtained from $Q,\widetilde{Q}$ using constant biquandle structures defined by $f,\tilde{f}$, respectively. Then ${\rm Aut}(B) \leq {\rm Aut}(\widetilde{B})$.
\end{corollary}
\begin{proof}
The proof follows from Proposition \ref{biquandle-auto-in-lifting} and \cite[Corollary 4.2]{Horvat}.
\end{proof}

\end{document}